\documentclass[11pt]{amsart}

\usepackage{amsopn, amssymb, amscd, amsmath, amsthm, MnSymbol, stmaryrd}
\usepackage{enumerate}
\usepackage{color}
\usepackage [all] {xy}
\usepackage[perpage,symbol*]{footmisc}
\usepackage{array}
\usepackage{multirow}
\usepackage{graphicx}
\usepackage{color, colortbl}
\usepackage{array, booktabs, arydshln, xcolor}
\usepackage{dsfont}
\usepackage{mathrsfs}
\usepackage{nth}

\newcommand{\nc}{\newcommand}

\nc{\gramU}{\textbf{\textsl{U}}}

\nc{\fg}{\mathfrak{f} }     \nc{\vg}{\mathfrak{v} }       \nc{\wg}{\mathfrak{w} }
\nc{\zg}{\mathfrak{z} }     \nc{\ngo}{\mathfrak{n} }      \nc{\kg}{\mathfrak{k} }
\nc{\ngoc}{\widehat{\mathfrak{n}} }
\nc{\mg}{\mathfrak{m} }     \nc{\bg}{\mathfrak{b} }       \nc{\ggo}{\mathfrak{g} }
\nc{\ggoc}{\widehat{\mathfrak{g}} }
\nc{\sog}{\mathfrak{so} }
\nc{\sug}{\mathfrak{su} }   \nc{\spg}{\mathfrak{sp} }     \nc{\slg}{\mathfrak{sl} }
\nc{\glg}{\mathfrak{gl} }   \nc{\cg}{\mathfrak{c} }       \nc{\rg}{\mathfrak{r} }
\nc{\hg}{\mathfrak{h} }     \nc{\tg}{\mathfrak{t} }       \nc{\ug}{\mathfrak{u} }
\nc{\dg}{\mathfrak{d} }     \nc{\ag}{\mathfrak{a} }       \nc{\pg}{\mathfrak{p} }
\nc{\agc}{\widehat{\mathfrak{a}} }
\nc{\sg}{\mathfrak{s} }     \nc{\affg}{\mathfrak{aff} }
\nc{\ggob}{\overline{\mathfrak{g}} }

\nc{\pca}{\mathcal{P}}       \nc{\nca}{\mathcal{N}}       \nc{\lca}{\mathcal{L}}
\nc{\oca}{\mathcal{O}}       \nc{\mca}{\mathcal{M}}       \nc{\tca}{\mathcal{T}}
\nc{\aca}{\mathcal{A}}       \nc{\cca}{\mathcal{C}}       \nc{\gca}{\mathcal{G}}
\nc{\sca}{\mathcal{S}}       \nc{\hca}{\mathcal{H}}       \nc{\bca}{\mathcal{B}}
\nc{\dca}{\mathcal{D}}       \nc{\rca}{\mathcal{R}}

\nc{\val}{\operatorname{val}}

\nc{\vp}{\varphi}
\nc{\ddt}{\tfrac{{\rm d}}{{\rm d}t}}
\nc{\dpar}{\tfrac{\partial}{\partial t}}

\nc{\im}{\sqrt{-1}}        

\newcommand{\fna}[5]{
\begin{array}{rccl}
{#1}:&\hspace{-2mm}{#2}&\hspace{-2mm}\longrightarrow&\hspace{-2mm}{#3}\\
     &\hspace{-2mm}{#4}&\hspace{-2mm}\longmapsto    &\hspace{-2mm}{#5}
\end{array}}

\newcommand{\fnn}[3]{
\begin{array}{rccl}
{#1}:&\hspace{-2mm}{#2}&\hspace{-2mm}\longrightarrow&\hspace{-2mm}{#3}\\
\end{array}}

\newcommand{\comillas}[1]{\textquotedblleft{#1}\textquotedblright}

\nc{\SO}{\mathrm{SO}}           \nc{\Spe}{\mathrm{Sp}}          \nc{\Sl}{\mathrm{SL}}
\nc{\SU}{\mathrm{SU}}           \nc{\Or}{\mathrm{O}}            \nc{\U}{\mathrm{U}}
\nc{\Se}{\mathrm{S}}            \nc{\Cl}{\mathrm{Cl}}           \nc{\Spein}{\mathrm{Spin}}
\nc{\Pin}{\mathrm{Pin}}
\nc{\Glr}{\mathrm{GL}_n(\RR)}   \nc{\Glc}{\mathrm{GL}_n(\CC)}   \nc{\Glv}{\mathrm{GL}(V)}    \nc{\Glk}{\mathrm{GL}_n(\fk)}   \nc{\Gl}{\mathrm{GL}}
\nc{\GrpG}{\mathrm{G}}          \nc{\GrpH}{\mathrm{H}}          \nc{\GrpA}{\mathrm{A}}       \nc{\GrpT}{\mathrm{T}}          \nc{\GrpK}{\mathrm{K}}
\nc{\GrpGc}{\widehat{\mathrm{G}}}
\nc{\GrpN}{\mathrm{N}}

\nc{\g}{\mathfrak{gl}_n(\RR)}

\nc{\RR}{{\Bbb R}} \nc{\HH}{{\Bbb H}} \nc{\CC}{{\Bbb C}} \nc{\ZZ}{{\Bbb Z}}
\nc{\FF}{{\Bbb F}} \nc{\NN}{{\Bbb N}} \nc{\QQ}{{\Bbb Q}} \nc{\PP}{{\Bbb P}}

\nc{\euler}{{\rm e}}

\nc{\vs}{\vspace{.2cm}} \nc{\vsp}{\vspace{1cm}}

\nc{\ip}{\langle \cdot , \cdot \rangle}
\nc{\ipd}{\langle \hspace{-0.5mm}\langle \cdot , \cdot \rangle\hspace{-0.5mm}\rangle}
\nc{\ippd}{( \hspace{-0.5mm} ( \cdot , \cdot ) \hspace{-0.5mm} )}
\nc{\ipp}{(      \cdot , \cdot       )}
\nc{\la}{\langle} \nc{\ra}{\rangle}

\nc{\ortsum}{ \mbox{\tiny $\displaystyle \bigoplus^{\perp}$}}

\nc{\dirsum}{ \mbox{\tiny $\displaystyle \bigoplus $}}

\nc{\ortres}{ \mbox{\tiny $\displaystyle \bigominus^{\perp}$}}

\newcommand{\ipa}[2]{\langle {#1} , {#2} \rangle}

\newcommand{\ipda}[2]{\langle \hspace{-0.5mm}\langle {#1} , {#2} \rangle\hspace{-0.5mm}\rangle}

\nc{\unm}{\tfrac{1}{2}}\nc{\unc}{\tfrac{1}{4}} \nc{\und}{\tfrac{1}{16}}

\nc{\no}{\vs\noindent}

\nc{\lamn}{\Lambda^2(\RR^n)^*\otimes\RR^n} \nc{\lamp}{\Lambda^2\pg^*\otimes\pg}
\nc{\lamg}{\Lambda^2\ggo^*\otimes\ggo} \nc{\lamngo}{\Lambda^2\ngo^*\otimes\ngo}
\nc{\lamnk}{\Lambda^2(\fk^n)^*\otimes \fk^n} \nc{\lamnkt}{\Lambda^3(\fk^n)^*\otimes \fk^n}

\nc{\tangz}{{\rm T}^{\rm Zar}}
\nc{\mum}{/\!\!/} \nc{\kir}{/\!\!/\!\!/}
\nc{\lievark}{\mathfrak{L}_n(\fk)}         \nc{\lievarc}{\mathfrak{L}_n(\CC)}        \nc{\lievarr}{\mathfrak{L}_n(\RR)}
\nc{\solvvark}{\mathfrak{R}_n(\fk)}        \nc{\solvvarc}{\mathfrak{R}_n(\CC)}       \nc{\solvvarr}{\mathfrak{R}_n(\RR)}
\nc{\nilvark}{\mathfrak{N}_n(\fk)}         \nc{\nilvarc}{\mathfrak{N}_n(\CC)}        \nc{\nilvarr}{\mathfrak{N}_n(\RR)}
\nc{\cirre}{\textrm{C}}
\nc{\fk}{\mathrm{k}}

\nc{\Ri}{\tfrac{4\Ric_{\mu}}{||\mu||^2}}

\nc{\ds}{\displaystyle}

\nc{\lb}{[\cdot,\cdot]}

\nc{\Hess}{\operatorname{Hess}}
\nc{\diag}{\operatorname{Diag}}   \nc{\Id}{\operatorname{Id}}
\nc{\trans}{\mbox{{\tiny$\operatorname{T}$}}}                         \nc{\Proj}{\operatorname{Proj}}
\nc{\Proy}{\operatorname{Proy}}
\nc{\ad}{\operatorname{ad}}       \nc{\Ad}{\operatorname{Ad}}        
\nc{\rank}{\operatorname{rank}}   \nc{\codim}{\operatorname{codim}}  
\nc{\Irr}{\operatorname{Irr}}     \nc{\End}{\operatorname{End}}
\nc{\Aut}{\operatorname{Aut}}     \nc{\Inn}{\operatorname{Inn}}
\nc{\lRad}{\operatorname{Rad}}
\nc{\Der}{\operatorname{Der}}     \nc{\Ker}{\operatorname{Ker}}
\nc{\Iso}{\operatorname{I}}       \nc{\Diff}{\operatorname{Diff}}
\nc{\Lie}{\operatorname{Lie}}     \nc{\tr}{\operatorname{tr}}
\nc{\dif}{\operatorname{d}}       \nc{\e}{\operatorname{e}}
\nc{\sen}{\operatorname{sen}}     \nc{\tang}{\operatorname{T}}
\nc{\modu}{\operatorname{mod}}
\nc{\Riem}{\operatorname{Rm}}     \nc{\Ric}{ {Ric}}
\nc{\sym}{\operatorname{sym}}     \nc{\symac}{\operatorname{sym^{ac}}}   \nc{\symc}{\operatorname{sym^{c}}}
\nc{\scalar}{\operatorname{sc}}
\nc{\grad}{\operatorname{grad}}
\nc{\ricci}{\operatorname{ric}}   \nc{\nr}{\operatorname{nr}}            \nc{\riccic}{\operatorname{ric^{c}}}
\nc{\riccig}{\operatorname{ric^{\gamma}}}
\nc{\Rin}{\operatorname{M}}
\nc{\Kill}{\operatorname{B}}
\nc{\Le}{\operatorname{L}}
\nc{\level}{\operatorname{level}} \nc{\rad}{\operatorname{r}}
\nc{\abel}{\operatorname{ab}}
\nc{\CH}{\operatorname{CH}}        \nc{\mcc}{\operatorname{mcc}}     \nc{\inte}{\operatorname{int}}         \nc{\aff}{\operatorname{Aff}}
\nc{\CaC}{\operatorname{CC}}        \nc{\ccm}{\operatorname{ccm}}
\nc{\Adj}{\operatorname{Adj}}
\nc{\Order}{\operatorname{O}} \nc{\Ricg}{\operatorname{Ric^{\gamma}}}
\nc{\Hom}{\operatorname{Hom}}
\nc{\sign}{\operatorname{sign}}
\nc{\spanv}{\operatorname{span}}
\nc{\xp}{\operatorname{xp}}   \nc{\xt}{\operatorname{xt}}
\nc{\IC}{\operatorname{IC}}   \nc{\OC}{\operatorname{OC}}

\nc{\rhov}{\operatorname{\rho_{v}}}
\nc{\mm}{m}
\nc{\mmt}{\widetilde{m}}
\nc{\F}{\operatorname{F}}

\theoremstyle{plain}
\newtheorem{theorem}{Theorem}[section]
\newtheorem{proposition}[theorem]{Proposition}
\newtheorem{corollary}[theorem]{Corollary}
\newtheorem{lemma}[theorem]{Lemma}

\theoremstyle{definition}
\newtheorem{definition}[theorem]{Definition}
\newtheorem{notation}[theorem]{Notation}

\theoremstyle{remark}
\newtheorem{remark}[theorem]{Remark}

\newtheorem{example}[theorem]{Example}

\numberwithin{equation}{section}

\begin{document}

\title{On distinguished orbits of reductive representations}
\author{EDISON ALBERTO FERN\'ANDEZ CULMA}
\address{Current affiliation: CIEM, FaMAF, Universidad Nacional de C\'ordoba, \newline \indent Ciudad Universitaria, \newline \indent (5000) C\'ordoba, \newline \indent Argentina}
\email{efernandez@famaf.unc.edu.ar}
\thanks{Fully supported by a CONICET fellowship (Argentina)}

\begin{abstract}
Let $\GrpG$ be a real reductive Lie group and let ${\tau}:{\GrpG} \longrightarrow \mathrm{GL}({V})$ be a real reductive representation of $\GrpG$ with (restricted) moment map $\mm_{\ggo}:V\smallsetminus\{ 0\} \longrightarrow \ggo$. In this work, we introduce the notion of \textit{nice space} of a real reductive representation to study the problem of how to determine if a $\GrpG$-orbit is \textit{distinguished} (i.e. it contains a critical point of the norm squared of $\mm_{\ggo}$). We give an elementary proof of the well-known convexity theorem of Atiyah-Guillemin-Sternberg in our particular case and we use it to give an easy-to-check sufficient condition for a $\GrpG$-orbit of a element in a nice space to be distinguished. In the case where $\GrpG$ is algebraic and $\tau$ is a rational representation, the above condition is also necessary (making heavy use of recent results of M. Jablonski), obtaining a generalization of Nikolayevsky's nice basis criterium. We also provide useful characterizations of nice spaces in terms of the weights of $\tau$. Finally, some applications to ternary forms and minimal metrics on nilmanifolds are presented.
\end{abstract}

\maketitle

\section{Introduction}

Let $\GrpGc=\U^{\CC}$ be a complex reductive group, $\widehat{\ggo}:= \Lie(\GrpGc)$ and ${\widehat{\tau}}:{\GrpGc} \longrightarrow \mathrm{GL}({\widehat{V}})$ be a representation of $\GrpGc$. In \cite{NESS1}, Linda Ness associates to $\widehat{\tau}$ a moment map (as in symplectic geometry) $\mm_{\widehat{\ggo}}:\widehat{V}\smallsetminus\{ 0\} \longrightarrow \sqrt{-1}\ug$ to study the orbit space of $\widehat{\tau}$ (here $\ug:=\Lie(\U)$). By considering a (real) inner product on $\sqrt{-1}\ug$ such that $\U$ acts by isometries via the adjoint representation, Ness's studies of the functional $||\mm_{\widehat{\ggo}}||^{2}$ not only recovered known results about the set of semi-stable vectors but she also proved that all critical points of $||\mm_{\widehat{\ggo}}||^2$ in the \textit{null cone} have similar properties to those of \textit{minimal points} (which have closed $\GrpGc$-orbit). The gradient flow of $||\mm_{\widehat{\ggo}}||^{2}$ was also studied independently by Francis Kirwan \cite{KIRWAN2} and the stratification given by such flow is often called the \textit{Kirwan-Ness stratification}.

Numerous efforts have been made over the next 25 years to carry out Kirwan-Ness results to the context of real Lie groups. This has been successfully achieved by Peter Heinzner, Gerald Schwarz and collaborators within their project of developing a geometric invariant theory for actions of real Lie groups on complex spaces (see \cite{HEINZNER1, HEINZNER3}). At this point, we should also mention the pioneering work of Roger Richardson and Peter Slodowy \cite{RICHARDSON1} on representations of real reductive algebraic groups, which was picked up by Patrick Eberlein and Michael Jablonski to extend \cite{KIRWAN2, NESS1} to this particular case (see \cite{EBERLEIN1, JABLONSKI4}). The idea of all the mentioned works is to use some connection with the complex case. In both cases, orbits of critical points of $||\mm_{\ggo}||^2$ play a distinguished role in the orbit space and it is for this reason that such orbits are called \textit{distinguished}.

Our purpose in this paper is to study the problem of how to determine if a $\GrpG$-orbit is distinguished, where $\GrpG$ is real reductive Lie group (in the sense of Heinzner-Schwarz) and the action is via a real reductive representation of $\GrpG$ (as it is defined in \cite{STOTZEL1}). In addition to the aforementioned motivation to study distinguished orbits, a second reason comes from the intriguing interplay between the Ricci flow on nilpotent Lie groups and the gradient flow of the norm squared of the moment map associated to the natural action of $\Glr$ on $V=\lamn$; the distinguished orbits of nilpotent Lie brackets are in 1-1 correspondence with \textit{nilsoliton metrics} on simply connected nilpotent Lie groups (see, for instance, \cite{LAURET6}). Nikolayevsky proves in \cite{NIKOLAYEVSKY2} many theorems on \textit{Einstein nilradicals} (nilpotent Lie algebras admitting a nilsoliton metric) by using the results given in \cite{RICHARDSON1}. Among these results, we can highlight the Nikolayevsky nice basis criterium, which provides an easy-to-check convex geometry condition for a nilpotent Lie algebra with a \textit{nice basis} to admit a nilsoliton metric.

By following Nikolayevsky's ideas given in \cite{NIKOLAYEVSKY1, NIKOLAYEVSKY2}, we introduce the notion of \textit{nice space} of a real reductive representation and we show that such criterium is a general fact of the theory of real reductive representations which are \textit{rational}. To do this, we give an elementary proof of the convexity of $\mm_{\ag}(\GrpA\cdot v)$ where $\GrpA$ is a connected abelian Lie group with out compact factor acting linearly and by symmetric operators on a finite dimensional real vector space $V$. This result is related with the convexity theorem of Atiyah-
Guillemin-Sternberg (\cite[Theorem 2]{ATIYAH1} and \cite[Theorem 5.2]{GUILLEMIN1})and indeed can be used to prove such theorem in the particular case of representations of a complexified torus.

In Section \ref{seccionsabernice}, we generalize some of the results of \cite{LAURETW2} and give an elementary characterization of a nice space in terms of weights of the representation (which is very helpful, as we will see). In Section \ref{seccionaplicaciones} we give some applications of our results to the study of ternary forms and the existence problem of minimal compatible metrics with geometric structures on nilpotent Lie groups (as are defined by Lauret in \cite{LAURET3}). Theorem \ref{straternary} provides an elementary expression for the \textit{stratifying set} associated with the natural action of $\mathrm{GL}_{3}(\RR)$ on ternary forms and Theorem \ref{3-4} gives a complete classification of distinguished orbits in the null cone of the $\mathrm{GL}_{3}(\RR)$-action on $\RR[x,y,z]_4$. Theorem \ref{nicegeneralizado} follows immediately from our main result and provides a tool to find minimal compatible metrics for a very wide family of class-$\gamma$ nilpotent Lie groups; we use it to determine minimal compatible metrics on symplectic two-step Lie algebras of dimension $6$ (see Theorem \ref{symtwo} and Table \ref{minmetricas}).

\section{Preliminaries}

Let $\GrpG$ be a real reductive Lie group in the sense of Heinzner-Schwarz, i.e. there exists a complex reductive group $\GrpGc=\U^{\CC}$ ($\GrpGc$ is the \textit{universal complexification} of a compact subgroup $\U$) such that $\GrpG$ is a closed subgroup of $\GrpGc$ and is \textit{compatible with the Cartan decomposition} $\GrpGc=\U\exp(\sqrt{-1}\ug)$, in other words, the function
\begin{eqnarray}\label{cartanDEC}
  \fna{\varphi}{\GrpK \times \pg }{ \GrpG }{(k,X)}{k\exp(X)}
\end{eqnarray}
is a diffeomorphism on $\GrpG$, where $\GrpK:=\GrpG \cap \U$ and $\pg := \ggo \cap \sqrt{-1}\ug$.

\begin{definition}\cite[Section 2]{STOTZEL1}
A representation  ${\tau}:{\GrpG} \longrightarrow \mathrm{GL}({V})$ of a real reductive Lie group is called \textit{real reductive representation} if $V$ is a $\GrpG$-invariant real subspace of a holomorphic $\GrpGc$-representation space $\widehat{V}$.
\end{definition}

It is fairly easy to see that given a real reductive representation, there exists a inner product on $V$, say $\ip$, such that $\GrpK$ acts by isometries and $\pg$ acts by symmetric operators. There also exists a inner product on $\ggo$, say $\ipd$, such that $\GrpK$ acts by isometries and $\pg$ acts by symmetric operators (via the adjoint representation) and $\ggo = \kg \, \ortsum \, \pg$ (it is a orthogonal decomposition). From now on, $\pi:=\dif \tau |_{e}:\ggo \longrightarrow \mathfrak{gl}(V)$ and orthogonal complements and orthogonal projections on $\ggo$ and $V$ are considered with respect $\ipd$ and $\ip$, respectively.

Let us fix a subalgebra $\ag$ of $\ggo$ maximal in $\pg$. Since $[\pg,\pg]\subseteq \kg$, it follows that $\ag$ is abelian. Thus, the families $\{\ad(X) | X \in \ag\}$ and $\{\pi(X) | X\in \ag\}$ are families of commuting symmetric operators, which allow a weight space decomposition for $\ggo$ and $V$; that is, there exist finite subsets $\Delta(\ggo)$ and $\Delta(V)$ of $\ag$, with $0 \nin \Delta(\ggo)$, such that
\begin{eqnarray}
\label{decGGO}  \ggo &=& \ggo_0 \, \ortsum \bigoplus^{\perp}_{\gamma \in \Delta(\ggo) } \ggo_{\gamma} , \\
\label{decV}    V &=& \bigoplus^{\perp}_{\alpha \in \Delta(V) } V_{\alpha}
\end{eqnarray}
where
\begin{eqnarray}
\nonumber \ggo_{\gamma}&=&\{Y\in \ggo : \ad(X)Y=\ipda{X}{\gamma}Y \mbox{ for all } X\in\ag \},\\
\nonumber  V_{\alpha}&=&\{v\in V : \pi(X)v=\ipda{X}{\alpha} v \mbox{ for all } X\in\ag\}.
\end{eqnarray}
The Decomposition (\ref{decGGO}) is often called \textit{restricted-root space decomposition} of $\ggo$ and the set $\Delta(\ggo)$ is called \textit{set of roots} of $\ggo$

Now, we consider the Cartan involution of $\ggo$ given by
\begin{equation}\label{resInvCartan}
\fna{\theta}{\kg \, \ortsum \, \pg}{\kg \,\ortsum \, \pg}{X+Y}{X-Y}, \, \forall X\in\kg, Y\in\pg.
\end{equation}
The following proposition summarizes the basic well-known properties of $\theta$, $\ip$ and $\ipd$ so we omit the proof.

\begin{proposition}\label{resumen}
Let $\GrpG$ be a real reductive Lie group, Let ${\tau}:{\GrpG} \longrightarrow \mathrm{GL}({V})$ be a real reductive representation with $\ip$ and $\ipd$ as above. Then, for all $X, Y$ in $\ggo$,
\begin{enumerate}
  \item $\pi(X)^{\trans} = -\pi(\theta(X))$  where $\pi(X)^{\trans}$ is the transpose operator of $\pi(X)$ with respect to $\ip$.
  \item $\ad(X)^{\trans} = -\ad(\theta(X))$ where $\ad(X)^{\trans}$ is the transpose operator of $\ad(X)$ with respect to $\ipd$.
  \item\label{descOrtgREAL} $\theta$ is an isometry of $\ipd$.
  \item $\theta(\ggo_{\lambda})=\ggo_{\theta(\lambda)}=\ggo_{-\lambda}$. Hence, $\lambda \in \Delta(\ggo)$ if and only if $-\lambda \in \Delta(\ggo)$.
  \item $[\ggo_{\lambda_1},\ggo_{\lambda_2}] \subseteq \ggo_{\lambda_1 + \lambda_2}$. Hence, if $[\ggo_{\lambda_1},\ggo_{\lambda_2}] \neq \{0\}$ then $\lambda_1 + \lambda_2 \in \Delta(\ggo)\bigcupdot \{0\}$.
  \item\label{sumaraizpesoR} $\pi(\ggo_{\lambda})V_{\alpha} \subseteq V_{\lambda + \alpha}$. Hence, if $\pi(\ggo_{\lambda})V_{\alpha} \neq \{0\}$ then $\lambda + \alpha \in \Delta(V)$.
 \item For all $Z \in \ggo_{\lambda}$, $[\theta(Z),Z] \in \ag$, thus $[\theta(Z),Z]=||Z||^2\lambda$.
\end{enumerate}
\end{proposition}

By fixing a \textit{regular element} $X_0\in \ag$ (i.e. $\ipda{\gamma}{X_0}\neq 0$ for every $\gamma \in \Delta(\ggo)$), we obtain a set of positive roots
\begin{equation}\label{proots}
\Delta(\ggo)^{+}=\{\gamma \in \Delta({\ggo}) : \ipda{\gamma}{X_0}>0\}
\end{equation}
and a Weyl chamber
\begin{equation}\label{wchamber}
\ag^{+}=\{Y\in \ag : \ipda{\gamma}{Y}>0 \mbox{ for every } \gamma \in \Delta({\ggo})^{+}\}.
\end{equation}
If $\Delta(\ggo)^{-}:=-\Delta(\ggo)^{+}$, then $\Delta(\ggo)=\Delta(\ggo)^{+} \bigcupdot \Delta(\ggo)^{-}$ and $\theta(\Delta(\ggo)^{+})=\Delta(\ggo)^{-}$.

We are now in a position to define the moment map of $\tau$. This map is implicitly defined by
\begin{equation}\label{momentmap}
\begin{array}{rcl}
  \mm_{\ggo}:V\setminus \{0\}&\longrightarrow &\ggo  \\
\ipda{\mm_{\ggo}(v)}{X} &=& \frac{1}{||v||^2}\ipa{\pi(X)v}{v},
\end{array}
\end{equation}
for all $X\in\pg$ and $v\in V\smallsetminus\{ 0\}$.
It is easy to see that $\mm_{\ggo}$ is $\GrpK$-equivariant (i.e. $\mm_{\ggo}(k\cdot v)=\Ad(k)(\mm_{\ggo}(v))$ for all $k\in \GrpK$ and $v \in V\smallsetminus\{ 0\}$) and $\RR^{*}$-invariant (i.e. $\mm_{\ggo}(\lambda v)=\mm_{\ggo}(v)$ for all $\lambda \in \RR^{*}$ and $v \in V\smallsetminus\{ 0\}$). Also, we note that the image of $\mm_{\ggo}$ is contained in $\pg$, which is clear from the $\GrpK$-invariance of $\ip$.

Set $\F_{\ggo}=||\mm_{\ggo}||^2$, which we shall call the \textit{norm squared of the moment map}. As we said above, the critical points of $\F_{\ggo}$ play a important role in the study of the orbit space. We thus have the following definition.

\begin{definition}\cite[Definition 2.6]{JABLONSKI2}
 An orbit $\GrpG \cdot v$ is said to be \textit{$\GrpG$-distinguished} if $\GrpG \cdot v$ has a critical point of $\F_{\ggo}$.
\end{definition}

\begin{theorem}\label{stratification}$\mbox{ }$\hfill
\begin{enumerate}
\item \label{teocrit}\cite[Corollaries 6.10, 6.11]{HEINZNER3}
  For each $v\in V \smallsetminus\{ 0\}$, the following conditions are equivalent:
\begin{enumerate}
  \item $v$ is a critical point of $\F_{\ggo}$.
  \item $\pi(\mm_{\ggo}(v))v = \lambda v$ for some $\lambda \in \RR$.
  \item $v$ is a global minimum of $\F_{\ggo}|_{\GrpG\cdot v}$.
\end{enumerate}
Under these conditions, the set of critical points of $\F_{\ggo}$ lying in $\GrpG \cdot v$ equals $\GrpK \cdot v$ (up to scaling).
\item\label{defestratos} \cite[Theorem 7.3 y Corollary 7.6]{HEINZNER3} (\textbf{Kirwan-Ness Stratification}) The gradient flow of $\F_{\ggo}$ determines a stratification of $V \setminus \{0\}$, it is to say that
        $$
        V \setminus \{0\} = \bigcupdot_{\beta \in \overline{\ag^{+}}}\mathscr{S}_{\beta},
        $$
         where a \textit{stratum} $\mathscr{S}_{\beta}$ is given by
        $$
        \mathscr{S}_{\beta} = \left\{v \in  V \setminus \{0\} : \lim_{t \to -\infty} s_{v}(t) \in \mathscr{C}(\beta)\right\}.
        $$
        Here, $s_{v}(t)$ is the integral curve through $v$ of the gradient field of $\F_{\ggo}$ and $\mathscr{C}(\beta)$ is the set
        $$
        \mathscr{C}(\beta)=\{v \in V \setminus \{0\} : v \mbox{ is a critical point of } \F_{\ggo} \mbox{ and } \mm_{\ggo}(v) \in \Ad(\GrpK)\beta\}.
        $$
        The strata satisfy
        $$
        \overline{\mathscr{S}_{\beta}} \subseteq {\mathscr{S}_{\beta}} \, \cupdot \, \{ \mathscr{S}_{\beta '} \, : \, ||\beta '|| > ||\beta ||\}.
        $$
    \item \cite[Theorem 5.4]{HEINZNER3} For all $v \in \mathscr{S}_{\beta} $, $\overline{\GrpG \cdot v} \cap \mathscr{S}_{\beta}$ contains a unique distinguished orbit.
\end{enumerate}
\end{theorem}

Let us denote by $\mathscr{B}$ the set of all $\beta \in \overline{\ag^{+}}$ such that $\mathscr{S}_{\beta} \neq \phi$ or, what is the same thing,
$$
\mathscr{B} = \{\beta \in \overline{\ag^{+}} : \mathscr{C}(\beta) \neq \phi \}.
$$
We call $\mathscr{B}$ the \textit{stratifying set} of $V$.

\begin{notation}
Let $\Phi$ be a finite subset of $\ag$. The convex hull of $\Phi$ will be denoted by $\CH (\Phi)$ while by $\aff(\Phi)$ we denote the affine space generated by $\Phi$. $\mcc (\Phi)$ denotes the \textit{minimal convex combination} of $\Phi$; i.e. the unique vector closest to the origin in $\CH (\Phi)$. The notation $\inte(\CH(\Phi))$ represents the interior of $\CH(\Phi)$ relative to the usual topology of $\aff(\Phi)$.

Fix $v \in V\smallsetminus\{ 0\}$, say $v=v_{1} + \ldots + v_{s}$ with each $v_i \in V_{\alpha_i}$ (by Equation (\ref{decV})). Let $\mathfrak{R}(v)$ denote the ordered set of weights related with $v$:
\begin{equation}\label{WRel}
    \mathfrak{R}(v)=\{ \alpha_i \in \Delta({V}) : v_i\neq 0 \}
\end{equation}
 and set $\beta_{v}=\mcc(\mathfrak{R}(v))$
\end{notation}\label{stratumdef}

We note that the set $\{\beta_v : v\in  V\smallsetminus\{ 0\}  \}$ is a finite set.

\begin{proposition}\label{propfinalCAP1}$\mbox{ }$\hfill
\begin{enumerate}
\item \label{proyMM}\label{mmsubcompatible}Let $\GrpH$ be a subgroup of $\GrpG$ compatible whit the Cartan decomposition of $\GrpG$ given by Equation (\ref{cartanDEC}). Then $\mm_{\hg}=\Proy_{\hg} \circ \mm_{\ggo}$, where $\Lie(\GrpH)=\hg$ and $\Proy_{\hg}:\ggo \longrightarrow \hg$ is the orthogonal projection of $\ggo$ on $\hg$ with respect to $\ipd$.
\item \label{criticocap} Let $v \in V\setminus \{0\}$ such that $\mm_{\ggo}(v)=\beta \in \ag$. Then
    \begin{enumerate}
      \item $\mm_{\ggo}(v) \in \inte(\CaC(\mathfrak{R}(v)))$.
      \item $v$ is a critical point of $\F_{\ggo}$ if and only if $\beta=\mcc(\mathfrak{R}(v))$.
    \end{enumerate}
\item $\mathscr{B}$ is a finite subset of $\ag$.
\end{enumerate}
\end{proposition}

\section{Generalization of Nikolayevsky's nice basis criterium}

In this section, we formulate and prove a generalization of Nikolayevsky's nice basis criterium (\cite[Theorem 3.]{NIKOLAYEVSKY2}). We begin with an elementary proof of the convexity of $\mm_{\ag}(\GrpA \cdot v)$ where $\GrpA$ is a connected abelian Lie group without compact factor acting linearly by symmetric operators on a real vector space $V$ (with respect some inner product $\ip$ on $V$).

\begin{proposition}\label{mmAorbit}
For every $v\in V\smallsetminus\{ 0\}$, $v=v_1+\ldots+v_s$ (as in Decomposition (\ref{decV}))
\begin{equation*}
    \mm_{\ag}(\GrpA \cdot v) = \inte(\CH(\mathfrak{R}(v))).
\end{equation*}
\end{proposition}

\begin{proof}
Fix $X\in \ag$. Let us first find an expression for $\mm_{\ag}(\exp(X)\cdot v)$ in terms of weights. Let $Y\in \ag$
\begin{eqnarray*}
  \ipda{\mm_{\ag}(\exp(X)\cdot v)}{Y} &=& \frac{1}{||\exp(X)\cdot v||^2}\ipa{\pi(Y)\exp(X)\cdot v}{\exp(X)\cdot v} \\
    &=& \frac{1}{||\exp(X)\cdot v||^2} \ipa{\sum \e^{\ipda{X}{\alpha_i}}\pi(Y)v_i}{\sum \e^{\ipda{X}{\alpha_j}}v_j}\\
    &=& \frac{1}{||\exp(X)\cdot v||^2} \sum \e^{2\ipda{X}{\alpha_i}}||v_i||^2\ipda{Y}{\alpha_i}\\
    &=& \ipda{\frac{1}{||\exp(X)\cdot v||^2}\sum \e^{2\ipda{X}{\alpha_i}}||v_i||^2 \alpha_i}{Y}.
\end{eqnarray*}
This gives
\begin{equation}\label{mmAform}
\mm_{\ag}(\exp(X)v) = \frac{1}{\sum \e^{2\ipda{X}{\alpha_i}}||v_i||^2}\sum\e^{2\ipda{X}{\alpha_i}}||v_i||^2 \alpha_i
\end{equation}
and so $\mm_{\ag}(\GrpA \cdot v) \subseteq \inte(\CH(\mathfrak{R}(v)))$.

We next prove that $\mm_{\ag}( \GrpA \cdot v)$ is a open convex set in $\aff(\mathfrak{R}(v))$. Let $\bg \subseteq \ag$ be a subspace of $\ag$ and let $\rho_{\bg}$ be given by
\begin{equation*}
    \fna{\rho_{\bg}}{\bg}{\RR}{X}{\ln||\exp(X)\cdot v||^2}.
\end{equation*}
By using that $\bg$ is abelian, a straightforward computation gives
$$
\mm_{\bg}(\exp(Y)\cdot v) = \frac{1}{2}(\dif \rho_{\bg})_{Y}.
$$
Therefore the image of $\dif \rho_{\bg}$ equals $2 \mm_{\bg}(\mathrm{B} \cdot v)$ ($\mathrm{B}=\exp(\bg)$). The basic idea of the proof is to apply the Fenchel's Convexity Theorem (see Appendix) to certain $\rho_{\widetilde{\bg}}$ and then to relate $\mm_{\ag}(\GrpA \cdot v)$ with $\dif \rho_{\widetilde{\bg}}$. We need to calculate $\dif^2 \rho_{\bg}$ for any $\bg \subseteq \ag$.  Let $X$, $Y \in \bg$ and $\widetilde{v}=\exp(X)\cdot v$
\begin{eqnarray*}
  \ipda{\dif^2 \rho_{\bg}(X) Y}{Y} &=& 2 \left.\ddt\right|_{0}\ipda{\mm_{\bg}(\exp(X+tY)\cdot v)}{Y} \\
   &=& 2 \left.\ddt\right|_{0} \frac{1}{||\exp(tY)\cdot \widetilde{v}||^2}\ipa{\pi(Y)\exp(tY)\cdot \widetilde{v}}{\exp(tY)\cdot\widetilde{v}} \\
   &=& \frac{2}{||\widetilde{v}||^4}\left[ (\ipa{\pi(Y)\pi(Y)\widetilde{v}}{\widetilde{v}}+\ipa{\pi(Y)\widetilde{v}}{\pi(Y)\widetilde{v}})||\widetilde{v}||^2 \right.\\
   & & \left. \hspace{1.2cm}-2\ipa{\pi(Y)\widetilde{v}}{\widetilde{v}}\ipa{\pi(Y)\widetilde{v}}{\widetilde{v}} \right] \\
   &=& \frac{4}{||\widetilde{v}||^4}[||\pi(Y)\widetilde{v}||^2||\widetilde{v}||^{2}-{\ipa{\pi(Y)\widetilde{v}}{\widetilde{v}}}^{2}].
\end{eqnarray*}
By Cauchy-Schwartz Inequality, $\ipda{\dif^2 \rho_{\bg}(X) Y}{Y} \geq 0$ and the equality holds if and only if $\pi(y)\widetilde{v}=\lambda \widetilde{v}$ for some $\lambda \in \RR$, which is equivalent to $\pi(y)v=\lambda v$ for some $\lambda \in \RR$. Thus $\dif^2 \rho_{\bg}(X)$ is positive definite for any $X\in\bg\smallsetminus\{ 0\}$ if and only if $\bg \cap \dg=\{0\}$ where
\begin{eqnarray*}
\dg:=\dg_{v}(\ag) &=& \{Y\in \ag : \pi(Y)v=\lambda v \mbox{ for some } \lambda\in\RR \}.
\end{eqnarray*}
Thus, let $\widetilde{\bg}:=\dg^{\perp}$ where $\dg^{\perp}$ is the orthogonal complement of $\dg$ in $\ag$.

We consider
\begin{eqnarray*}
\ag_v &=& \{Y \in \ag : \pi(Y)v=0\}.
\end{eqnarray*}
It is easy to see that $\ag_v$ is an ideal of $\dg$ of codimension $1$ or $0$; we have two cases
\begin{eqnarray}
\label{case1} \dg &=& \ag_v, \\
\nonumber  &or& \\
\label{case2} \dg &=& \ag_v \, \ortsum \, \RR Z .
\end{eqnarray}
with $Z\in \ag$ such that $\pi(Z)v=v$.

Now, let $X\in \ag$. We note that $\mm_{\ag}(\exp(X)v) = \mm_{\ag}(\exp(X')v)$ where $X'$ is the component of $X$ in $\widetilde{\bg}$. In fact, say $X = X'+X^{\prime \prime}$ with $X' \in \widetilde{\bg}$ and $X^{\prime \prime} \in \dg$,
\begin{eqnarray}
\nonumber   \mm_{\ag}(\exp(X)v) &=& \mm_{\ag}(\exp(X^{\prime})\exp(X^{\prime \prime})v ) \\
\nonumber    &=&  \mm_{\ag}(t\exp(X^{\prime})v) \mbox{ for some } t\in\RR\\
\label{reducmma} &=&  \mm_{\ag}(\exp(X^{\prime})v).
\end{eqnarray}
Here, we have used again that $\ag$ is abelian and $\mm_{\ag}$ is  invariant under a rescaling. Also, it is easy to see that for any $X \in \ag$
\begin{eqnarray}\label{ortomm}
    \mm_{\ag}(\exp(X)v) \in \ag \, \ortres \, \ag_{v}.
\end{eqnarray}

We are in a position to describe $\mm_{\ag}(\GrpA \cdot v)$ in terms of $\dif \rho_{\widetilde{\bg}}$. By the Equation (\ref{ortomm}), $\mm_{\ag}$ could have component in $\RR Z$ if the Case (\ref{case2}) holds. In this case, such component is always $\frac{1}{||Z||^2}Z$:
\begin{eqnarray*}
   \ipda{\mm_{\ag}(\exp(X)v)}{Z} &=& \frac{1}{||\exp(X)v||^2} \ipa{\pi(Z)\exp(X)v}{\exp(X)v} \\
   &=&
   \frac{1}{||\exp(X)v||^2}\ipa{\exp(X)\pi(Z)v}{\exp(X)v} \\
   &=&
   \frac{1}{||\exp(X)v||^2}\ipa{\exp(X)v}{\exp(X)v} \\
   &=& 1.
\end{eqnarray*}
We recall that $\mm_{\widetilde{\bg}}$ is the projection of $\mm_{\ag}$ on $\widetilde{\bg}$ (Proposition \ref{propfinalCAP1}, item \ref{proyMM}), from this, if $X^{'}=\Proj_{\widetilde{\bg}} X$ for $X\in \ag$ we have
\begin{eqnarray}
\nonumber   \mm_{\ag}(\exp(X)v) &=& \mm_{\ag}(\exp(X^{\prime})v) \mbox{ (by Equation (\ref{reducmma}) )}\\
\nonumber   &=& \Proy_{\widetilde{\bg}}\mm_{\ag}(\exp(X^{\prime})v) + \Proy_{\dg}\mm_{\ag}(\exp(X^{\prime})v)\\
\nonumber   &=& \mm_{\widetilde{\bg}}(\exp(X^{\prime})v) + \Proy_{\dg}\mm_{\ag}(\exp(X^{\prime})v)\\
\label{casosmmabel}   &=& \left\{ \begin{array}{l} \mm_{\widetilde{\bg}}(\exp(X^{\prime})\cdot v), \mbox{ if } \dg=\ag_{v}\\
        \mm_{\widetilde{\bg}}(\exp(X^{'})\cdot v) + \frac{1}{||Z||^2}Z, \mbox{ if } \dg = \ag_{v} \, \ortsum \, \RR Z \end{array}\right.
\end{eqnarray}
Now (\ref{casosmmabel}) becomes $\mm_{\ag}(\GrpA \cdot v)=\mm_{\widetilde{\bg}}(\widetilde{\mathrm{B}}\cdot v)$ or $\mm_{\ag}(\GrpA \cdot v)=\mm_{\widetilde{\bg}}(\widetilde{\mathrm{B}}\cdot v)+\frac{1}{||Z||^2}Z$ (here, $\widetilde{\mathrm{B}} = \exp(\widetilde{\bg})$), and since  $\mm_{\widetilde{\bg}}(\widetilde{\mathrm{B}}\cdot v)$ is a convex set  ($\rho_{\widetilde{\bg}}$ satisfies the hypotheses of the Fenchel's Convexity Theorem), whatever the case, we conclude that $\mm_{\ag}(\GrpA \cdot v)$ is a convex set.

To finish the proof, it is sufficient to prove that $\overline{\mm_{\ag}(\GrpA \cdot v)} = \CaC(\mathfrak{R}(v))$ (by \cite[Theorem 2.1.7]{GRUNBAUM1}, we have for any convex subset $\Omega$ of $\RR^n$, $\inte(\overline{\Omega})=\inte(\Omega)$). By Theorem \ref{expTh} item \ref{politope} we could prove that each exposed point of $\CaC(\mathfrak{R}(v))$ belongs to $\overline{\mm_{\ag}(\GrpA \cdot v)}$ . Let $\alpha_i$ be an exposed point of $\CaC(\mathfrak{R}(v))$. By definition, there exist $H_{i} \in \ag$ and $h_{i} \in \RR$ such that for each $X \in \CaC(\mathfrak{R}(v))$, $\ipda{X}{H_i}\geq h_i$ and the equality holds if and only if $X=\alpha_i$. By using this inequality in the Equation (\ref{mmAform}) it follows that
\begin{eqnarray*}
    \ds \lim_{t \to{-}\infty} \mm_{\ag}(\exp(t H_i)\cdot v) = \alpha_i.
\end{eqnarray*}
On account of the above, we have $\mm_{\ag}(\GrpA \cdot v)$ is a convex set such that $\overline{\mm_{\ag}(\GrpA \cdot v)}=\CaC(\mathfrak{R}(v))$, and, in consequence, $\mm_{\ag}(\GrpA \cdot v)=\inte(\CaC(\mathfrak{R}(v)))$.
\end{proof}

\begin{remark}
The above proposition is related with well-known results on convexity properties of the moment map (see, for instance \cite[Proposition 3]{HEINZNER2}). The argument used in the proof shows that it is independent of complex case (which is a difference with analogous results). In fact, by a similar argument, we can prove the respective result in the complex setting, this is, if $\GrpT$ is a torus and $\GrpT^{\CC}$ its universal complexification which acts linearly on complex vector space $\widehat{V}$ then $\mm_{\tg^{\CC}} (\GrpT^{\CC}\cdot v) = \inte(\CaC(\mathfrak{R}(v)))$ for any $v \in \widehat{V}\setminus \{0\}$ ($\mathfrak{R}(v)$ is defined analogously as above). Convexity properties of moment maps is currently a well-developed theory and has a rich history which includes results due to Schur, Horn, Kostant, Mumford, Atiyah, Guillemin and Sternberg, Kirwan, etc. (see \cite{HEINZNER4}).
\end{remark}

\begin{definition}\label{defnicespace}
Let ${{\tau}}:{\GrpG} \longrightarrow \mathrm{GL}({{V}})$ be a real reductive representation and let $W$ be a vector subspace of $V$. $W$ is called a \textit{nice space} if it is $\GrpA$-invariant and $\mm_{\ggo}(w)\in \ag$ for any $w \in W \setminus \{0\}$ (here, $\GrpA =\exp(\ag)$). Elements of $W$ are called \textit{nice elements}.
\end{definition}

\begin{remark}
Given that $\GrpA$ is compatible with the Cartan decomposition of $\GrpG$, if $W$ is a nice space, then $\mm_{\ag}(w)=\mm_{\ggo}(w)$ for any $w \in W \setminus \{0\}$.
\end{remark}

\begin{remark}
One could define a nice element $w\in V$ as that which satisfies $\mm_{\ggo}(\GrpA \cdot w) \subseteq \ag$. In the Proposition \ref{niceelement}, we prove that this notion coincides with that given in Definition \ref{defnicespace}. \end{remark}

\begin{remark}
In \cite{JABLONSKI2} there is not a name to the notion of nice space but certainly this should be \textit{$\GrpA$-detectable}. Since our motivation comes from the  Nikolayevsky's nice basis criterium, we prefer to use name \comillas{nice} in this definition.
\end{remark}

Now, we are ready to give a first step in the proof of our generalization of Nikolayevsky's nice basis criterium.

\begin{proposition}\label{onedirection}
 Let $W$ be a nice space and let $w\in W \setminus \{0\}$. If $\mcc(\mathfrak{R}(w)) \in \inte(\CaC(\mathfrak{R}(w)))$, then the $\GrpG$-orbit of $w$ is distinguished. Furthermore, there exists $X \in \ag$ such that $\exp(X)w$ is a critical point of $\F_{\ggo}$.
\end{proposition}

\begin{proof}
We suppose that $\mcc(\mathfrak{R}(w)) \in \inte(\CH(\mathfrak{R}(w)))$. By Proposition \ref{mmAorbit} there exists $X\in \ag$ such that $\mm_{\ag}(\exp(X)\cdot w)=\mcc(\mathfrak{R}(w))$ and moreover, $\F_{\ag}|_{\GrpA \cdot w}$ has a minimum value at $\widetilde{w}=\exp(X)\cdot w$. By Theorem \ref{stratification} item \ref{teocrit}, $\widetilde{w}$ is a critical point of $\F_{\ag}$ and so $\pi(\mm_{\ag}(\widetilde{w}))\widetilde{w}=\lambda \widetilde{w}$ for some $\lambda \in \RR$. Since $w$ is nice and so is $\widetilde{w}$, $\mm_{\ag}(\widetilde{w})=\mm_{\ggo}(\widetilde{w})$. From this and, again, from Theorem \ref{stratification} item \ref{teocrit}, $\widetilde{w}$ is a critical point of $\F_{\ggo}$, and so $\GrpG \cdot w$ is a distinguished orbit.
\end{proof}

\begin{notation}\label{Gram}
Let $w$ be a nice element and consider the ordered set $\mathfrak{R}(w)$. We denote by $\gramU_{w}$ the Gram matrix of $(\mathfrak{R}(w),\ipd)$, i.e.
\begin{equation}\label{Gramform}
   \gramU_{w}(p,q)=\ipda{\mathfrak{R}(w)_{p}}{\mathfrak{R}(w)_{q}}
\end{equation}
with $1\leq p,q \leq \sharp\mathfrak{R}(w) $
\end{notation}

We omit the proof of the following corollaries, which follows easily from the method of Lagrange multipliers.

\begin{corollary}\label{coro1}
  Let $w$ be a nice element. If the equation
\begin{equation}\label{matrixU1}
    \gramU_{w}[x_i]=\lambda[1]
\end{equation}
has a solution $[x_i]$ for some $\lambda \in \RR$ such that each $x_i$ is positive (\textit{positive solution}) then $\GrpG \cdot w$ is a distinguished orbit.
\end{corollary}

\begin{corollary}\label{coro2}
  Let $w$ be a nice element and suppose that $0 \nin \CH(\mathfrak{R}(w))$. If the equation
\begin{equation}\label{matrixU2}
    \gramU_{w}[x_i]=[1]
\end{equation}
has a positive solution $[x_i]$ then $\GrpG \cdot w$ is a distinguished orbit.
\end{corollary}

To obtain the proof of the converse of Proposition \ref{onedirection} we need to restrict ourselves to the algebraic case. The idea is to apply Jablonski's results given in \cite{JABLONSKI2}. Let us recall some notions and results of \cite{JABLONSKI2} which we need.

\begin{definition}\cite[Definition 2.9]{JABLONSKI2}\label{jabdetecting}
Let $\GrpG$ be a (real or complex) linear reductive algebraic group acting linearly and rationally on a finite dimensional vector space $V$. Let $\GrpH$ be a compatible subgroup of $\GrpG$ and $W$ be an $\GrpH$-stable smooth subvariety of $V$. We say that the $\GrpG$-action is \textit{$\GrpH$-detectable along $W$} if $\mm_{\ggo}(w) \in \hg$ for any $w\in W$.
\end{definition}

\begin{theorem}\cite[Corollary 3.4]{JABLONSKI2}\label{detectingth}
Let $\GrpG$, $\GrpH$, $V$ and $W$ be as in the above definition and assume that $W$ is a closed subset of $V$. If the action is $\GrpH$-detectable along $W$, then for any $w\in W$, $\GrpG\cdot w$ is $\GrpG$-distinguished if and only if $\GrpH\cdot w$ is $\GrpH$-distinguished.
\end{theorem}

\begin{proposition}\label{onedirection2}
Let $\GrpG$ and $V$ as in the above definition and $W$ a nice space, and let $w\in W \setminus \{0\}$ such that $\GrpG \cdot w$ is a distinguished orbit. Then
$$
\mcc(\mathfrak{R}(w)) \in \inte(\CaC(\mathfrak{R}(w))).
$$
\end{proposition}

\begin{proof}
Since $\GrpG \cdot w$ is a distinguished, $\GrpA \cdot w$ is a distinguished orbit of the $\GrpA$-action (by Theorem \ref{detectingth}). Thus, there exist $\widetilde{w} \in \GrpA \cdot w$ which  is a critical point of $\F_{\ag}$. By Theorem \ref{stratification} item \ref{teocrit} we have $||\mm_{\ag}(\widetilde{w})||^2 \leq ||\mm_{\ag}(w^{\prime})||^2$ for any $w^{\prime} \in \GrpA \cdot w$. It follows from Proposition \ref{mmAorbit} that $\mm_{\ag}(\widetilde{w})$ is a vector of minimal norm in the convex set $\inte(\CaC( \mathfrak{R}(w)))$ and by a continuity argument, $\mm_{\ag}(\widetilde{w})$ is also a vector of minimum norm in the set $\CaC( \mathfrak{R}(w))$; hence $\mcc(\mathfrak{R}(w)) \in \inte(\CaC(\mathfrak{R}(w)))$.
\end{proof}

We have from Propositions \ref{onedirection} and \ref{onedirection2} our main result and equivalences in the Corollaries \ref{coro1} and \ref{coro2} to the algebraic case.

\begin{theorem}
 Let $\GrpG$ be a real linear reductive algebraic group and $\tau:\GrpG \longrightarrow \Glv$ be a rational representation of $\GrpG$, and let $w$ be a nice element. The orbit $\GrpG \cdot w$ is distinguished if and only if
$$
\mcc(\mathfrak{R}(w)) \in \inte(\CaC(\mathfrak{R}(w))).
$$
 Moreover, in such case, there exists $X \in \ag$ such that $\exp(X)w$ is a critical point of $\F_{\ggo}$.
\end{theorem}

\begin{remark}
It is easy to see that the above results are valid in the complex case.
\end{remark}

\section{How to know if your $\GrpA$-space is nice}\label{seccionsabernice}

In this section, we consider an $\GrpA$-invariant subspace $W$ of $V$ and its decomposition in weight spaces
\begin{equation}\label{Wdec}
    W = W_{\alpha_{1}} \, \ortsum  \, \ldots \, \ortsum W_{\alpha_{r}}
\end{equation}
with weight set $\Delta({W})=\{\alpha_1, \ldots, \alpha_r\}$. We follow ideas go back at least as far as \cite[Lemma 7.1]{SJAMAAR1} and results presented in \cite{LAURETW2}.

\begin{remark}
It is easily seen that the representation $\fnn{\pi|_{\ag}}{\ag}{\glg(V)}$ is a completely reducible representation of $\ag$. Hence, $\Delta({W})$ is a subset of $\Delta(V)$ and each $W_{\alpha_{i}}$ is a subspace of some $V_{\alpha_{j}}$. For this reason, there is no problem in considering $\mathfrak{R}(w)$ with respect to the decomposition (\ref{Wdec}) for all $w\in W$.
\end{remark}

\begin{lemma}\label{lema1}
Let $\gamma \in \Delta(\ggo)$, $Y \in \ggo_{\gamma}$, $\alpha_i \in \Delta(W)$ and $w_i \in W_{\alpha_{i}}$ such that $\pi(Y)w_{i} \neq 0$, then $\pi(Y)w_{i}$ is a simultaneous eigenvector for the family $\{\pi(X) : X \in \ag \}$. Moreover, if $\Proj_{W} \pi(Y)w_i \neq 0$ then $\Proj_{W} \pi(Y)w_i$ is a weight vector with corresponding weight $\gamma + \alpha_i$, and in consequence, $\gamma + \alpha_i = \alpha_j$ for some $\alpha_j \in \Delta(W)$ with $\alpha_i \neq \alpha_j$.
\end{lemma}

\begin{proof}

Since $Y \in \ggo_{\gamma}$, $[X,Y]=\ipda{\gamma}{X}Y$ for all $X\in \ag$ and thus
\begin{equation*}
    \ipda{\gamma}{X}\pi(Y) = \pi([X,Y]) = \pi(X)\pi(Y)-\pi(Y)\pi(X)
\end{equation*}

and finally
\begin{eqnarray*}
  \pi(X)(\pi(Y)w_i) & = & \ipda{\gamma}{X}\pi(Y)w_i+\pi(Y)\pi(X)w_i \\
                    & = & \ipda{\gamma}{X}\pi(Y)w_i + \ipda{\alpha_i}{X} \pi(Y)w_i \\
                    & = & \ipda{\gamma + \alpha_i}{X}\pi(Y)w_i .
\end{eqnarray*}

To see the second part, let $\pi(Y)w_i=w+w^{'}$ with $w\in W$ ($w=\Proj_{W}\pi(Y)w_i$) and $w^{'} \in W^{\perp}$.
\begin{eqnarray*}
  \pi(X)w + \pi(X)w^{'} & = & \pi(X)\pi(Y)w_i \\
                        & = & \ipda{\gamma + \alpha_i}{X}\pi(Y)w_i \\
                        & = & \ipda{\gamma + \alpha_i}{X}w + \ipda{\gamma + \alpha_i}{X}w^{'} .
\end{eqnarray*}

By using that $W^{\perp}$ is also $\GrpA$-invariant as $W$, we have $\pi(X)w^{'} \in W^{\perp}$, we thus get $\pi(X)w = \ipda{\gamma + \alpha_i}{X}w$.
\end{proof}

\begin{corollary}\label{critpesos}
  Fix a weight space $W_{\alpha_{i}}$ of the decomposition (\ref{Wdec}). $W_{\alpha_{i}}$ is a nice space where all its points are critical points of $\F_{\ggo}$ with $\mm_{\ggo}(w_i)=\alpha_i$ for all $w_i \in W_{\alpha_{i}}$.
\end{corollary}

\begin{proof}
We must prove that $\mm_{\ggo}(w_i) \in \ag$. By Equation (\ref{momentmap}) and the Decomposition (\ref{decGGO}), this is equivalent to show $\ipa{\pi(Y)w_i}{w_i}=0$ for all $\gamma \in \Delta(\ggo)$ and $Y \in \ggo_{\gamma}$.

Let $w=\Proj_{W} \pi(Y)w_i$, so $\ipa{\pi(Y)w_i}{w_i}=\ipa{w}{w_i}$. If $w=0$,  there is nothing to prove. In other case, by the previous lemma, $w$ is a weight vector of weight $\gamma + \alpha_i$. Since $\gamma \neq 0$, $\gamma + \alpha_i \neq \alpha_i$ which gives $w \perp w_i$.

From Expression (\ref{mmAform}), we have $\mm_{\ggo}(w_i)=\alpha_i$ and so $\pi(\mm_{\ggo}(w_i))w_i=\pi(\alpha_i)w_i=\ipda{\alpha_i}{\alpha_i}w_i$; $w_i$ is a critical point of $\F_{\ggo}$
\end{proof}

\begin{lemma}\label{lema2}
  Suppose that for all $\alpha_i, \, \alpha_j \in \Delta(W)$ such that $\alpha_j - \alpha_i \in \Delta(\ggo)$ one has that $\pi(\ggo_{\gamma}) W_{\alpha_{i}} \perp W$, where $\gamma=\alpha_j - \alpha_i$. Then,
\begin{equation*}
    \pi(\ggo_{\gamma}) W_{\alpha_{i}} \perp W,
\end{equation*}
for any $\gamma \in \Delta(\ggo)$, and $\alpha_i \in \Delta(W)$.
\end{lemma}

\begin{proof}
Suppose the assertion of the lemma is false. From Lemma \ref{lema1} we could find $\gamma \in \Delta(\ggo)$, $Y \in \ggo_{\gamma}$, $\alpha_i, \, \alpha_j \in \Delta(W)$ and $w_i \in W_{\alpha_{i}}$ such that $\Proj_{W} \pi(Y)w_i$ is a weight vector of weight $\gamma + \alpha_i = \alpha_j$. This gives $\alpha_j - \alpha_i \in \Delta(\ggo)$ and the hypothesis now becomes $\Proj_{W} \pi(Y)w_i = 0$; this is a contradiction.
\end{proof}

\begin{theorem}\label{nicebasis}
An $\GrpA$-invariant subspace of $W$ of $V$ is nice if and only if for all $\alpha_i, \, \alpha_j \in \Delta(W)$ such that $\alpha_j - \alpha_i \in \Delta(\ggo)$,
$$
\pi(\ggo_{\gamma}) W_{\alpha_{i}} \perp W
$$
where $\gamma=\alpha_j - \alpha_i$
\end{theorem}

\begin{proof}
 We suppose that $W$ is nice and let $\alpha_i$, $\alpha_j$ $\in \Delta(W)$ be such that $\alpha_j - \alpha_i \in \Delta(\ggo)$, say $\alpha_j - \alpha_i = \gamma$ with $\gamma\in \Delta(\ggo)$. We must prove that $\Proj_{W} \pi(Y)w_i=0$ with $Y \in \ggo_{\gamma}$ and $w_i\in W_{\alpha_{i}}$. Since $\Proj_{W} \pi(Y)w_i \in W_j$ (from Lemma \ref{lema1}), it will thus be sufficient to prove that
 \begin{equation*}
    \ipa{\Proj_{W} \pi(Y)w_i}{w_j}=0
 \end{equation*}
for any $w_j \in W_j$.
Consider a $w_j \in W_j$ and set $w=w_i+w_j$. Since $W$ is nice, $\mm_{\ggo}(w)\in \ag$ and using the Decomposition (\ref{decGGO}) we have
\begin{eqnarray}
\nonumber 0 &=& \ipa{\mm_{\ggo}(w)}{Y} \\
\nonumber &=& \ipa{\pi(Y)(w_i+w_j)}{w_i+w_j}\\
\nonumber &=& \ipa{\pi(Y)w_i}{w_i}+\ipa{\pi(Y)w_i}{w_j}\\
\nonumber & & +\ipa{\pi(Y)w_j}{w_i}+\ipa{\pi(Y)w_j}{w_j}\\
          &=& \ipa{\pi(Y)w_i}{w_j}+\ipa{\pi(Y)w_j}{w_i}
\end{eqnarray}

where we have used that $\ipa{\pi(Y)w_i}{w_i} = \ipa{\pi(Y)w_j}{w_j} = 0$ from Corollary \ref{critpesos}.

Note that $\ipa{\pi(Y)w_j}{w_i}=0$. Indeed, suppose, contrary to our claim, that $\ipa{\pi(Y)w_j}{w_i}\neq 0$. As in the proof of the Lemma \ref{lema2}, we have $\Proj_{W}\pi(Y)w_j$ is a weight vector of weight $\gamma + \alpha_j=\alpha_i$. But, $\gamma + \alpha_i = \alpha_j$ since $\alpha_j - \alpha_i = \gamma$, we now have a contradiction to that $\gamma \neq 0$.

According to the above, we have $0=\ipa{\pi(Y)w_i}{w_j}=\ipa{\Proj_{W}\pi(Y)w_i}{w_j}$

We now proceed to show necessary part. Let $w\in W$, say $w=\sum w_i$ and let $Y\in \ggo_{\gamma}$ with $\gamma \in \Delta(\ggo)$. We must see that $\ipa{\mm_{\ggo}(w)}{Y}=0$
\begin{eqnarray*}
  \ipa{\mm_{\ggo}(w)}{Y} &=& \frac{1}{||w||^2} \ipa{\pi(Y)w}{w}\\
                            &=& \frac{1}{||w||^2} \sum_{i,j} \ipa{\pi(Y)w_i}{w_j}.
\end{eqnarray*}
From Lemma \ref{lema2}, we have the proof.
\end{proof}

By the $\theta$-invariance of $\ipd$, the fact that $\theta(\ggo_{\gamma})=\ggo_{-\gamma}$ for each $\gamma \in \Delta(\ggo)$  and recalling that the image of the moment map is contained in $\pg$, we have:

\begin{corollary}\label{coro3}
 An $\GrpA$-invariant subspace $W$ of $V$ is nice if and only if for all $\gamma \in \Delta(\ggo)^{+}$
$$
\pi(\ggo_{\gamma}) W \perp W
$$
\end{corollary}

An obvious but very useful corollary is the following

\begin{corollary}\label{niceresta}
Let $W$ be a $\GrpA$-invariant subspace of $V$. If for all $\alpha_i$ and $\alpha_j$ in $\Delta(W)$, $\alpha_i-\alpha_j \notin \Delta(\ggo)$, then $W$ is nice.
\end{corollary}

We can now prove the equivalence between the two possible notions of nice element.

\begin{proposition}\label{niceelement}
Let $w = w_{i_{1}} + \ldots + w_{i_{s}}$ in $V$ with each $w_{i_{j}}$ a non-null vector in $V_{\alpha_{i_{j}}}$ (by Equation (\ref{decV})) and such that $\mm_{\ggo}(\GrpA \cdot w) \subset \ag$. Then $W = \RR w_{{i_{1}}}\, \ortsum \, \ldots \, \ortsum \, \RR w_{{i_{s}}}$ is a nice space.
\end{proposition}

\begin{proof}
From Corollary \ref{coro3}, we need to prove that $\ipa{\pi(\ggo_{\lambda})w_{i_{j}}}{w_{i_{k}}}=0$, for all $\lambda \in \Delta(\ggo)$ and $1 \leq j \leq k \leq s$. Suppose the proposition was false. Then we could find $\lambda \in \Delta(\ggo)$, $Y \in \ggo_{\lambda}$ and $j$ and $k$ such that
$$
\ipa{\pi(Y)w_{i_{j}}}{w_{i_{k}}}\neq 0.
$$
Reasoning as in the above results, we have $\lambda+\alpha_{i_{j}}=\alpha_{i_{k}}$, and so, the set
$$
\Omega=\{\alpha_{i_{j}} : \lambda + \alpha_{i_{j}} = \alpha_{i_{k}} \in \Delta(W) \mbox{ and } \ipa{\pi(Y)w_{i_{j}}}{w_{i_{k}}}\neq 0\}
$$
is nonempty. As $\ipda{\mm_{\ggo}(\GrpA \cdot w)}{Y}=0$, we have for all $X \in \ag$
\begin{eqnarray}
\nonumber  0 &=& \ipa{\pi(Y)\exp(X)w}{\exp(X)w} \\
\nonumber    &=&\sum_{j,k}\euler^{\ipda{X}{\alpha_{i_{j}}+\alpha_{i_{k}}}}\ipa{\pi(Y)w_{i_{j}}}{w_{i_{k}}}\\
\nonumber    &=&\sum_{\alpha_{i_{j}}\in \Omega}\euler^{\ipda{X}{\lambda+2\alpha_{i_{j}}}}\ipa{\pi(Y)w_{i_{j}}}{w_{i_{k}}}\\
\label{equation0}  &=& \sum_{\alpha_{i_{j}}\in \Omega}\euler^{\ipda{X}{2\alpha_{i_{j}}}}\ipa{\pi(Y)w_{i_{j}}}{w_{i_{k}}}
\end{eqnarray}

We consider the convex hull of $\Omega$ and let $\alpha_{i_{m}}$ be an exposed point of such convex set. So, there exist $H \in \ag$ and $h\in \RR$ such that $\ipda{X}{H}\geq h$ for all $X\in \CH(\Omega)$ and the equality holds if and only if $X = \alpha_{i_{m}}$.

From Equation \ref{equation0} follows that
\begin{eqnarray*}
  0& = & \frac{1}{\euler^{\ipda{tH}{2\alpha_{i_{m}}}}}\sum_{\alpha_{i_{j}}\in \Omega}\euler^{\ipda{tH}{2\alpha_{i_{j}}}}\ipa{\pi(Y)w_{i_{j}}}{w_{i_{k}}},
\end{eqnarray*}
and by letting $t \to -\infty$, we obtain
\begin{eqnarray*}
  0 &=&  \ipa{\pi(Y)w_{i_{m}}}{w_{i_{n}}}
\end{eqnarray*}
which is contrary to $\alpha_{i_{m}} \in \Omega$.
\end{proof}

\begin{corollary}\label{ultimocor}
 Let $\alpha_i$ and $\alpha_j$ be such that $W=W_{\alpha_{i}}\ortsum W_{\alpha_{j}}$ is a nice space. Then $\beta:=\mcc(\{\alpha_i,\alpha_j\})$ defines a stratum, i.e. $\mathscr{S}_{\beta} \neq \phi$.
\end{corollary}

\begin{proof}
  If $\mcc(\{\alpha_i,\alpha_j\})$ is $\alpha_i$ or $\alpha_j$, we have from Corollary \ref{critpesos} the conclusion for this case. In the other case, $\mcc(\{\alpha_i,\alpha_j\})$ is in the interior of $\CH\{\alpha_i,\alpha_j\}$ and so, Proposition \ref{onedirection} and Theorem \ref{stratification} complete the proof.
\end{proof}

\section{Applications}\label{seccionaplicaciones}

\subsection{Ternary forms}

In this part, we want to discuss some applications of the above results to \textit{classical invariant theory}. In classical invariant theory one studies polynomials and their intrinsic and geometrical properties; we mean those properties which are unaffected by a change of variables and are purely geometric (for instance, multiplicities of roots). Such theory was a focal and major topic in the \nth{19} century and in the beginning of the \nth{20} century, and was strongly impacted by Hilbert's contributions (which were used by David Mumford to develop the modern geometric invariant theory).

We prove a result that describes the stratifying set for the natural action of $\mathrm{GL}_{3}(\RR)$ on ternary forms in a very simple way (see Theorem \ref{straternary}) and later we give a classification of distinguished orbits in the null cone of real ternary quartics (the results can easily be extended to the complex case).

Along this section, let $\GrpG := \Glr$, $\ggo := \Lie(G)= \glg_{n}(\RR)$ and we consider the Cartan decomposition of $\ggo$ given by $\sog(n) \ortsum \sym(n)$ ($\kg := \sog(n)$ and $\pg := \sym(n)$). Let us denote by $\ipd$ the usual inner product on $\ggo$
$$
\ipda{X}{Y}=\tr(X Y^{\trans}),\, \forall X, Y \in \ggo.
$$
Let $\ag=\{(a_1 \ldots a_n):=\diag (a_1 \ldots a_n) : a_i \in \RR\}$ and let $\ag^{+}$ denote the usual Weyl chamber of $\mathfrak{gl}_{n}(\RR)$
$$
\ag^{+}=\{(a_1,\ldots,a_n  ) \in \ag : a_1 < \ldots < a_n \}.
$$
Thus, a restricted-root space decomposition for $\ggo$ is given for
\begin{equation*}
    \ggo = \ag \ortsum \bigoplus^{\perp}_{i \neq j} \RR E_{ij}
\end{equation*}
with $\Delta(\ggo) = \{\gamma_{ij}:=E_{ii}-E_{jj} : i \neq j\}$ ($\{E_{ij}\}_{1\leq i,j \leq n}$ is the canonical basis of $\glg_{n}(\RR)$).

Let $V:=\RR[x_1, \ldots , x_n]_{d}$, the vector space of all homogeneous polynomials of degree $d$ on $n$ variables (\textit{$n$-ary forms of degree $d$}) and we set $\ip$ the inner product such that monomials of $V$ are orthogonal and $||x_{1}^{d_1}... \, x_{n}^{d_n}||^2=d_{1}!\ldots d_{n}!$ ($\sum d_i =d)$. $\ip$ is easily seen to have the required conditions.

The action of $\GrpG$ on $V$ is given by linear change of variables,
\begin{equation*}
g \cdot p ( x_1, \ldots , x_n) = p \left( g^{-1} \left[\begin{array}{c} x_1 \\ \vdots \\ x_n \end{array}\right] \right),
\end{equation*}
for all $g \in \GrpG$ and $p \in V$. It follows easily that
\begin{equation*}
  \pi(E_{ij})p = \ddt|_{0} p( \e^{-tE_{ij}}\bullet) = -x_{j} \tfrac{\partial p}{\partial x_i},
\end{equation*}
for all $p \in V$. Hence, we obtain that the moment map to the action of $\mathrm{GL}_n(\RR)$ on $\RR[x_1,\ldots,x_n]_{d}$ is given by
$$
\mm_{\mathfrak{gl}_{n}(\RR)}(p)=\frac{-1}{||p||^2}\left(\ipa{x_j \frac{\partial p}{\partial x_i}}{p}\right), \, \forall p \in \RR[x_1,\ldots,x_n]_d \setminus \{0\}.
$$

The basis of weight vectors of $V$ are monomials and we have
\begin{equation*}
  \pi(Y)x_{1}^{d_1} \ldots x_{n}^{d_n} = - \left(\sum_{i=1}^{n } y_i d_i\right) x_{1}^{d_1} \ldots x_{n}^{d_n}
\end{equation*}
for all $Y=(y_1, \ldots, y_n)\in \ag$. Consequently, $\Delta(V) =\{ \alpha=-(d_1,...,d_n) \in \ag : \sum d_i = d\}$. In the sequel, given $\alpha \in \Delta(V)$, say $\alpha=-(d_1,\ldots,d_n)$ with $\sum d_i =d$, we will denote by $\textbf{x}^{\alpha}$ the monomial $x_1^{d_1}, \ldots , x_n^{d_n}$.

\begin{remark}
Since the identity matrix, $\Id$, which is in $\ag$, acts as $\pi(\Id)p=-p$ for any $p \in V$, from Theorem \ref{stratification} (\ref{teocrit}), it follows that critical points of $\F_{\mathfrak{gl}(n,\RR)}$ and $F_{\mathfrak{sl}(n,\RR)}$ are the same. Moreover, the respective stratifications coincide up to translation in $\mathscr{B}$ sets.
\end{remark}

We now consider the case of ternary forms (i.e. $n=3$). Binary forms (i.e. $n=2$) were studied by Ness in \cite[Lemma 10.4]{NESS1}, where she gives a description of non-minimal critical points of $\F_{\ggo}$.

First observe that all weights are in the equilateral triangle with vertices $-(d,0,0)$, $-(0,d,0)$ and $-(0,0,d)$. It follows that if $\Phi \subseteq  \Delta(V)$, the $\CH(\Phi)$ is contained in such triangle. In the sequel, $\beta_{0}$ will denote the barycenter of triangle; $\beta_{0}:=-(\frac{d}{3},\frac{d}{3},\frac{d}{3})$. Thus, we have that $\beta=\mcc(\Phi)$ if and only if $\beta$ is the unique vector in $\CH(\Phi)$ closest to $\beta_{0}$. Accordingly, $\mcc(\Phi)$ is either equal to $\beta_{0}$, or there are $\alpha_i,\, \alpha_j \in \Phi$ such that $\mcc(\Phi)=\mcc(\{\alpha_i,\alpha_j\})$.  We can also realize the barycenter as the minimal convex combination of two weights as follows:
\begin{equation*}
    \beta_{0} = \mcc(\{\alpha_i , \alpha_j\}) \mbox{ with } \left\{ \begin{array}{l} \alpha_i=-(k,k,0) \mbox{ and }\alpha_j=-(0,0,2k) \mbox{ if } d=2k, \\ \alpha_i=-(k,k,1) \mbox{ and }\alpha_j=-(0,0,2k+1) \mbox{ if } d=2k+1, \end{array}\right.
\end{equation*}
with the property that $W=\RR \textbf{x}^{\alpha_i} \, \ortsum \, \RR \textbf{x}^{\alpha_j}$ is a nice space. We will see that this holds for every $\beta$ such that $\mathcal{S}_{\beta}\neq \phi$.

\begin{lemma}
  Let $\alpha_i$, $\alpha_j$ in $\Delta(V)$ and $W=\spanv_{\RR}\{\textbf{\textsc{x}}^{\alpha_i},\textbf{\textsc{x}}^{\alpha_j}\}$. Then $W$ is a nice space if and only if $\alpha_j - \alpha_i \notin \Delta(\ggo)$.
\end{lemma}

\begin{proof}
  From Corollary \ref{niceresta}, it remains to prove the if part. Let $W$ be a nice space and contrary to our claim, that $\alpha_j - \alpha_i \in \Delta(\ggo)$. Let $\alpha_i=-(m_1,m_2,m_3)$ and $\alpha_j=-(n_1,n_2,n_3)$ and without loss of generality we can assume $\alpha_j - \alpha_i = \gamma_{12}= E_{11}-E_{22}=(1,-1,0) $. So, $m_1=1+n_1$, $m_2=n_2 -1$ and $m_3 = n_3$. Let $Y = E_{12}$, $Y\in \ggo_{\gamma_{12}}$
\begin{eqnarray*}
  \pi(Y)\textbf{x}^{\alpha_i} &=& \pi(E_{12})x^{1+n_1}y^{n_2-1}z^{n_3} \\
       &=&  (1+n_1)x^{n_1}y^{n_2}z^{n_3}\\
       &=&  (1+n_1)\textbf{x}^{\alpha_j}
\end{eqnarray*}
Therefore, $\Proj_{W} \pi(Y)\textbf{x}^{\alpha_j} \neq 0 $, which contradicts the Theorem \ref{nicebasis}.
\end{proof}

\begin{lemma}\label{ternlemma}
  Let $\beta \in \ag$ such that $\mathscr{S}_{\beta} \neq \phi$. Then there exist $\alpha_i$ and $\alpha_j$ in $\Delta(V)$ such that $W=\spanv_{\RR}\{\textbf{\textsc{x}}^{\alpha_i},\textbf{\textsc{x}}^{\alpha_j}\}$ is a nice space and $\beta=\mcc(\{\alpha_i,\alpha_j\})$. In particular, there exists $q\in W$ such that $q$ is a critical point of $\F_{\ggo}$ and $q\in \mathscr{S}_{\beta} $
\end{lemma}

\begin{proof}
  According to the above remark, there exist $\widetilde{\alpha_i}$ and $\widetilde{\alpha_j}$ in $\Delta(V)$ such that $\beta = \mcc(\{\widetilde{\alpha_i}, \widetilde{\alpha_j}\})$. If $\widetilde{\alpha_i}$ and $\widetilde{\alpha_j}$ satisfy the conclusion of the lemma, then we are done. If they do not, then from the previous lemma, $\widetilde{\alpha_j} - \widetilde{\alpha_i} \in \Delta(\ggo)$, say $\widetilde{\alpha_j} - \widetilde{\alpha_i} = (1,-1,0)$; this involves no loss of generality. As in the previous lemma, $\widetilde{\alpha_i}=-(1+n_1,n_2-1,n_3)$ and $\widetilde{\alpha_j}=-(n_1,n_2,n_3)$. It is easy to check that
 $\alpha_i=-(0,n_1+n_2,n_3)$ and $\alpha_j=-(n_1+n_2,0,n_3)$ are in $\aff(\{\widetilde{\alpha_i}, \widetilde{\alpha_j}\})$, and $\beta=\mcc(\{\alpha_i,\alpha_j\})$. In this case, the pair $\alpha_i$ and $\alpha_j$ proves the existence part, because we must have $n_1+n_2 > 1$; on the contrary, $\beta=-(\frac{1}{2},\frac{1}{2},n_3)$ and $\mathscr{S}_{\beta} = \phi$.
\end{proof}

\begin{lemma}
   If $\beta \in \{-(\frac{1}{2},\frac{1}{2},d-1), -(\frac{1}{2},d-1,\frac{1}{2}), -(d-1,\frac{1}{2},\frac{1}{2})\}$, then $\mathscr{S}_{\beta}$ is not a stratum.
\end{lemma}

\begin{proof}
  It suffices for our purposes to prove that $\mathscr{S}_{\beta}$ has not critical points of $\F_{\ggo}$. From Theorem \ref{stratification} (\ref{teocrit}), we can prove  that there exists no $q \in V$ such that $\mm_{\ggo}(q)=\beta$ and $\pi(\beta)q = ||\beta||^2 q$. If $q = \sum_{\alpha \in \Delta(V)} a_{\alpha} \textbf{x}^{\alpha}$, $\pi(\beta)q=||\beta||^2 q$ if and only if $q \in \RR \textbf{x}^{\alpha_{i_{1}}} \, \ortsum \, \ldots\, \ortsum \, \RR \textbf{x}^{\alpha_{i_{s}}}$ with $\ipda{{\alpha_{i_{j}}}}{\beta}=||\beta||^2$; i.e. ${\alpha_{i_{j}}}'$s are in the affine space
\begin{equation*}
    \Omega = \{X \in \ag : \ipda{X}{\beta} = ||\beta||^2\}.
\end{equation*}

$\Omega$ meets the equilateral triangle (that contains all weights) at 1-dimensional affine space, which has only two weights of $\Delta(V)$; namely $\alpha_i=(1,0,d-1)$ and $\alpha_j=(0,1,d-1)$. Thus, one such $q$ is of the form $axz^{d-1}+byz^{d-1}$. It is easily seen that $q$ is in the orbit of $p=cxz^{d-1}$ and since $p \in \mathscr{S}_{\beta^{'}}$ with $\beta^{'}=-(1,0,d-1)$, it follows that $\mathscr{S}_{\beta}=\phi$ by definition of stratum.
\end{proof}

By Corollary \ref{ultimocor} and Lemma \ref{ternlemma} we deduce the main result of this section.

\begin{theorem}\label{straternary}
Let $\mathscr{B}$ be as in the Theorem \ref{stratification} for the natural action of $\mathrm{GL}_{3}(\RR)$ on $\RR[x,y,z]_{d}$. Then
\begin{equation}
\mathscr{B} = \{\beta=\mcc(\{\alpha_i,\alpha_j\})  : \alpha_i, \, \alpha_j \in \Delta(V) \mbox{ with } \alpha_j - \alpha_i \nin \Delta(\ggo)\} \cap \overline{\ag^{+}}
\end{equation}
\end{theorem}

Although the Theorem \ref{straternary} makes very easy to find the set $\mathscr{B}$, the problem of finding the critical points of $\F_{\ggo}$ in each stratum is a very difficult task. For instance, the critical points of $\F_{\ggo}$ in $\mathscr{S}_{\beta_{0}}$ (recall that $\beta_{0} = -(\frac{d}{3},\frac{d}{3},\frac{d}{3})$) bring us to the classification of closed $\mathrm{SL}_3(\RR)$-orbits in $V$ (which is a very well-known open problem). Now, we explore such problem  in ternary quartics ($n=3$ and $d=4$, Figure \ref{triangle}).

\begin{figure}
 \includegraphics[width=9cm]{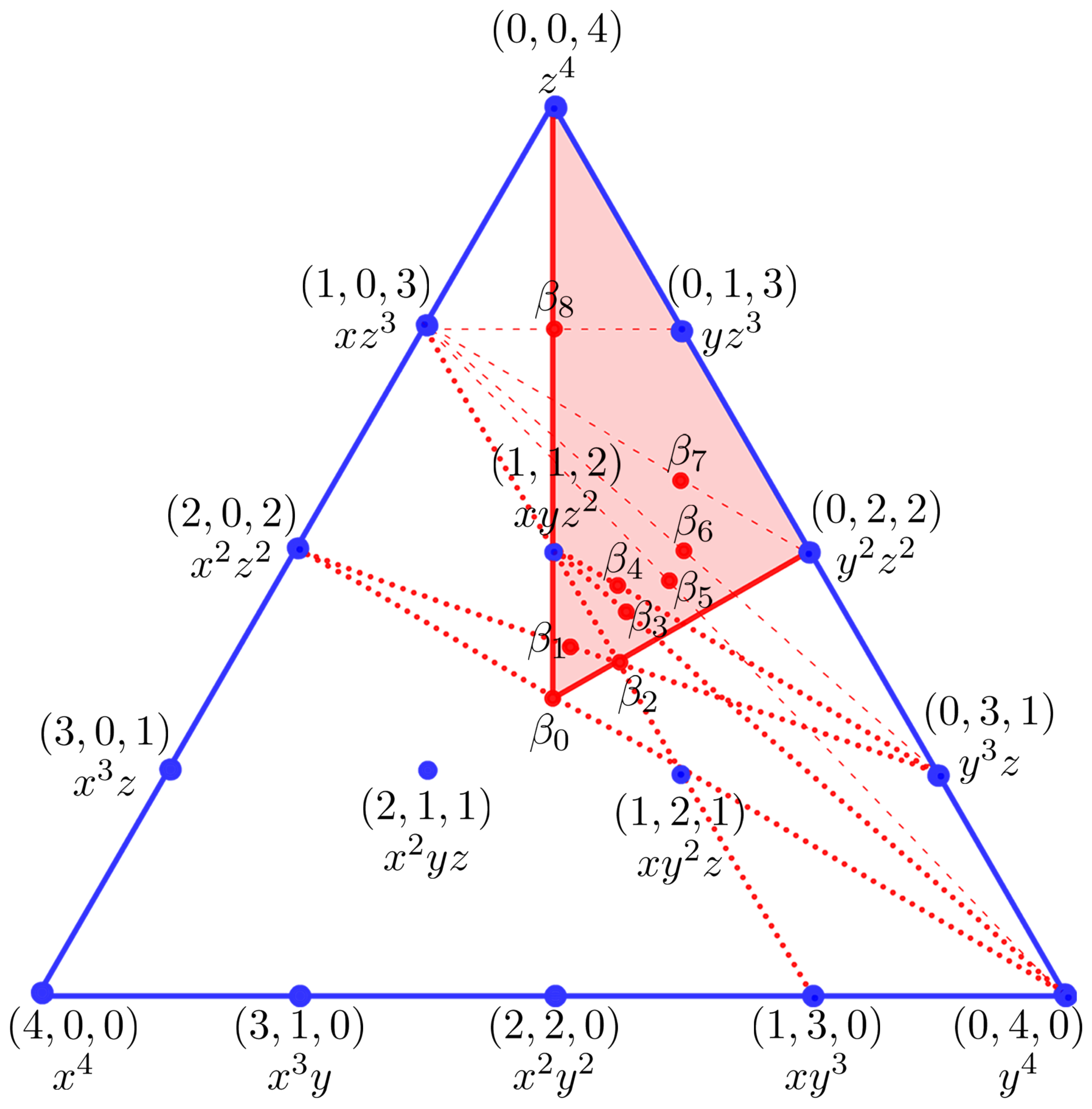}\\
  \caption{Equilateral triangle, $n=3$ and $d=4$}\label{triangle}
\end{figure}

By using the Theorem \ref{straternary}, we can find the set $\mathscr{B}$ in the following straightforward way. For each $\beta \in \mathscr{B}$, we calculate the vector subspace of $V$, $Z_{\beta}$ given by
\begin{equation*}
Z_{\beta}:= \left\{ p = \sum_{\alpha \in \Delta(V)} a_{\alpha} \textbf{x}^{\alpha} : \ipda{\alpha}{\beta}=||\beta||^2 \mbox{ with } a_{\alpha}\neq 0\right\}.
\end{equation*}
This subspace is associated to the affine space
\begin{equation*}
  \Omega(\beta) = \{ X \in \ag : \ipda{X}{\beta}=||\beta||^2\}
\end{equation*}
and we must find the weights in the intersection of $\Omega(\beta)$ with the equilateral triangle, which is a $k$-dimensional convex set with $0\leq k \leq 2$. Finally, we find
\begin{equation}\label{conjptocri}
  \{ p \in Z_{\beta} : \mm_{\ggo}(p)=\beta \mbox{ with }||p||=1\},
\end{equation}
which gives the set $\mathscr{C}(\beta)$ (as in Theorem \ref{stratification} (\ref{defestratos})).

The case $\beta_0$, as we said above, the critical points in $\mathscr{S}_{\beta_{0}}$ correspond to $\mathrm{SL}_3(\RR)$-closed orbits and are also called minimal (they are global minima of $\F_{\ggo}$).  In Table \ref{table}, we give the strata of ternary quartic forms. The third column in the table shows the form that a critical point should take according the $\beta$ given in the type column. For $\beta_2 = \left(1,\frac{3}{2},\frac{3}{2}\right)$, a critical point must be of the form $p=x(az^3+byz^2+cy^2z+dy^3)$. In this case, it is easy to see that $p$ is a critical point if and only if the $\mathrm{SL}_{2}(\RR)$-orbit of the binary form $\frac{p}{x}=az^3+byz^2+cy^2z+dy^3$ is closed (for instance, by comparing $\mm_{\mathfrak{sl}_{3}(\RR)}(p)$ with $\beta_2$). The same holds for critical points of type $(1,1,2)$ and $(0,2,2)$. In the first case, a critical point must be of the form $q=z^2(ax^2+bxy+cy^2)$ and so the binary form $\frac{q}{z^2}=ax^2+bxy+cy^2$ must has a $\mathrm{SL}_{2}(\RR)$-closed orbit. In the last case, we have a binary quartic form, as in the mentioned cases above, must has a $\mathrm{SL}_{2}(\RR)$-closed orbit. To study such cases, we use the well-known canonical forms for binary forms of low degree  (see, \cite[Chapter V: 22.4, 23.2, \S25-Exercises 13 and 14]{GUREVICH1} or \cite[Pag. 9, Pag. 28, Pag.30-Exercise 2.25]{OLVER1}); by using such classification, it is fairly easy to see those with $\mathrm{SL}_{2}(\RR)$-closed orbit. For the convenience of the reader, we recall in the appendix (Tables \ref{table4t}, \ref{table3n} and \ref{table2d}), the canonical forms for the real or complex binary forms of degree 4,3 y 2 respectively.

For the remaining cases, however, it is very simple to find all distinguished orbits by explicit calculation of the set in Equation (\ref{conjptocri}).

\begin{theorem}\label{3-4}
 The classification of distinguished orbits in the null cone of $\RR[x,y,z]_4$ for the natural action of $\mathrm{SL}_{3}(\RR)$ is given in Table \ref{table}.
\end{theorem}

\begin{table}[h]
  \centering
\begin{tabular}{|c|c|c|l|}
\hline
\multicolumn{1}{ | c         |  }  {\centering Nt.}          &
\multicolumn{1}{   c|           }  {\centering Type}         &
\multicolumn{2}{   c|           }  {\centering Critical point}\tabularnewline
\hline
\multirow{2}{*}{$\beta_0 $} & \multirow{2}{*}{$(\frac{4}{3},\frac{4}{3},\frac{4}{3})$}   & \multirow{1}{*}{ $\ds \sum_{\alpha \in \Delta(V)} a_{\alpha} \textbf{x}^{\alpha}$ } & \multirow{2}{*}{minimal}\\
 & & & \\
\hline
$\beta_1 $ & $(\frac{8}{7}, \frac{9}{7},\frac{11}{7})$ &  $ax^2z^2+by^3z$  & $a=\pm\sqrt{\frac{1}{7}}$, $b=\sqrt{\frac{1}{14}}$\\
\hline
\multirow{2}{*}{$\beta_2$}  & \multirow{2}{*}{$( 1, \frac{3}{2},\frac{3}{2})$} & \multirow{1}{*}{$x\left(az^3+bz^2y+\right.$ }  & $x[y^3 + y z^2]$ \\
\cline{4-4}
                            &                                                  & \multirow{1}{*}{$\left.czy^2+dy^3\right)$}   & $x[y^3 - y z^2]$ \\
\hline
$\beta_3$  & $(\frac{6}{7}, \frac{10}{7},\frac{12}{7})$ & $axyz^2+by^4$ & $a=\sqrt{\frac{3}{7}}$, $b=\pm \sqrt{\frac{1}{168}}$ \\
\hline
$\beta_4$  & $( \frac{5}{6}, \frac{8}{6},\frac{11}{6})$ & $axyz^2+by^3z$ & $a=\sqrt{\frac{5}{12}}$, $b=\sqrt{\frac{1}{36}}$\\
\hline
\multirow{3}{*}{$\alpha_{xyz^2}$ }    & \multirow{3}{*}{$( 1,1,2 )$}                                & \multirow{2}{*}{$z^2\left(ax^2+bxy+\right.$} &
 $z^2[x^2+y^2]$\\
\cline{4-4}
            &                                          &                                                                                &
 $-z^2[x^2+y^2]$\\
\cline{4-4}
            &                                                                             & \multirow{1}{*}{$\left.cy^2\right)$}                       &
$z^2[x^2-y^2]$\\
\hline
$\beta_5$  & $(\frac{8}{13} ,\frac{20}{13} , \frac{24}{13})$ & $axz^3+by^4$  & $a=\sqrt{\frac{4}{39}}$, $b=\pm \sqrt{\frac{5}{312}}$ \\
\hline
$\beta_6$  & $(\frac{1}{2} ,\frac{3}{2} , 2 )$               & $axz^3+by^3z$ & $a={\sqrt{\frac{1}{12}}}$, $b= {\sqrt{\frac{1}{12}}}$ \\
\hline
$\beta_7$  & $(\frac{1}{3} ,\frac{4}{3} , \frac{7}{3})$      & $axz^3+by^2z^2$ & $a=\sqrt{\frac{1}{18}}$, $b= \pm \sqrt{\frac{1}{6}}$\\
\hline
\multirow{3}{*}{$\alpha_{y^2z^2}$ }           & \multirow{3}{*}{$( 0 , 2 , 2)$}                                   & \multirow{3}{*}{$\sum_{i=0}^{4} a_i y^{4-i}z^{i}$} & $y^4+2ty^2z^2+z^4$ ($t \in \RR$) \\
\cline{4-4}
                             &                                  &
                             &  $-[y^4+2ty^2z^2+z^4]$ ($-1 \leq t$)\\
\cline{4-4}
                             &                                   &
                             &  $y^4+2ty^2z^2-z^4$ ($t \in \RR$) \\
\hline
$\beta_8$  & $ (\frac{1}{2} ,\frac{1}{2} , 3 )$              &  $axz^3+byz^3$                   &  \verb"empty"  \\
\hline
$\alpha_{yz^3}$           & $(0 , 1 , 3) $                                  & $ayz^3$                           & $a= \sqrt{\frac{1}{6}}$ \\
\hline
$\alpha_{z^4}$           & $(0 , 0 , 4 )$                                  & $az^4$                            & $a=\pm \sqrt{\frac{1}{24}}$\\
\hline
\end{tabular}
\caption{Classification of distinguished orbits in the null cone of $\RR[x,y,z]_4$ for the natural action of $\mathrm{GL}_{3}(\RR)$.}\label{table}
\end{table}

\begin{remark}
Although we did not classify closed $\mathrm{SL}_{3}(\RR)$-orbits, our approach provides different families of closed $\mathrm{SL}_{3}(\RR)$-orbits. For instance, we consider a maximal family $\Phi$ of weights which are not \textit{neighbors}, (by neighbors we mean those pairs of weights that have a rest in $\Delta(\mathfrak{gl}_n(\RR)$)) and such that $\beta_{0}\in\CH(\Phi)$. From Corollary \ref{niceresta} the generated space by $\Phi$ is nice and a polynomial $p$ such that $\mathfrak{R}(p)=\Phi$ must have closed $\mathrm{SL}_{3}(\RR)$-orbit (see Proposition \ref{onedirection}). One can prove that the family $p_{a,\ldots,f}(x,y,z)=ax^4+by^4+cz^4+dx^2y^2+ex^2z^2+fy^2z^2$ contains three-parametric families of closed $\mathrm{SL}_{3}(\RR)$-orbits.
\end{remark}

\subsection{Canonical compatible metrics for geometric structures on nilmanifolds}

In \cite{LAURET3}, Jorge Lauret noted that there is an intriguing relationship between the geometry of nilpotent Lie groups and the geometric invariant theory (GIT) applied to actions of reductive subgroups of $\mathrm{GL}_{m}(\RR)$ on $\Lambda^2(\RR^{m})^*\otimes\RR^{m}$. For instance, the recent advances in the study of \textit{Einstein solvmanifolds} and, more generally, \textit{Solvsolitons}, have come from using powerful tools that are given by GIT (see \cite{LAURET6, NIKOLAYEVSKY2}). By using this fact, it has been proposed in \cite{LAURET3} a way to study the problem of finding \comillas{the best metric} which is compatible with a fixed geometric structure on a simply connected nilpotent Lie group. In this approach, it is defined the notion of \textit{minimal compatible metric} and the properties that make a minimal metric \comillas{special} are proved; a minimal metric is unique (up to isometry and scaling) when it exists and it can be characterized as a \textit{soliton solution} of the \textit{invariant Ricci flow} (see \cite[Theorem 4.4]{LAURET3}).  A priori, such properties are far from obvious, however, such facts follow of the last-mentioned connection. We refer the interested reader to \cite{LAURET3} for a complete discussion on this approach.

It is our purpose in this section to study the natural action of $\Glr$ and its reductive subgroups on $V:=\Lambda^2(\RR^{n})^*\otimes\RR^{n}$ given by change of basis:
$$
g\cdot \mu(X,Y)=g \mu(g^{-1} X, g^{-1} Y), \, X,Y \in \RR^n, \, g\in\Glr, \, \mu \in V.
$$
The corresponding representation of $\g$ on $V$ is given by
$$
A\cdot\mu(X,Y)=A\mu(X,Y)-\mu(AX,Y)-\mu(X,AY), \, A\in\g \, \mu\in V,
$$
The inner product $\ipd$ on $\g$ and a Cartan decomposition of $\Glr$ is given in the above application and the inner product $\ip$ on $V$ is the induced inner product by the canonical inner product of $\RR^n$ (denoted also by $\ip$) as follows:
$$
\ipa{\mu}{\lambda}=\sum_{ijk}\ipa{\mu(e_i,e_j)}{e_k}\ipa{\lambda(e_i,e_j)}{e_k}, \, \forall \mu, \lambda \in V.
$$
An easy computation shows that the inner product $\ip$ on $V$ satisfies the required conditions.

Let $(\GrpN_{\mu},\gamma)$ be a \textit{class-$\gamma$ nilpotent Lie group}, i.e. $\GrpN_{\mu}$ is a simply connected nilpotent Lie group with Lie algebra $\ngo = (\RR^n,\mu)$ and $\gamma$ is an invariant geometric structure on $\GrpN$ satisfying \cite[Defintion 2.1]{LAURET3}. The main result of this section is a result of existence of minimal compatible metrics for a very wide family of class-$\gamma$ nilpotent Lie groups (see Theorem \ref{nicegeneralizado}). Let us recall basic definitions and results of \cite{LAURET3} and then to give the mentioned application.

\begin{definition}\cite[Definition 2.2]{LAURET3}
Let $\ipp$ be a \textit{compatible metric} with a class-$\gamma$ nilpotent Lie group $(\GrpN_{\mu},\gamma)$ (as in the \cite[Definition 2.1 (iii)]{LAURET3}). We consider the orthogonal projection $\Ric^{\gamma}_{\ipp}$ of the Ricci operator $\Ric_{\ipp}$ on $\ggo_{\gamma} = \Lie(\mathrm{G}_{\gamma})$ relative to the inner product $\ipd$. $\Ric^{\gamma}_{\ipp}$ is said to be \textit{invariant Ricci operator}, and the corresponding \textit{invariant Ricci tensor} is given by $\ricci^{\gamma}=\ipa{\Ric^{\gamma}\cdot}{\cdot}$.
\end{definition}

\begin{definition}[\textbf{Minimal compatible metric}]\cite[Definition 2.3]{LAURET3}
A left invariant metric $\ipp$ compatible with a class-$\gamma$ nilpotent Lie group $(\GrpN_{\mu},\gamma)$ is called \textit{minimal} if
$$
||\Ric^{\gamma}_{\ipp}||^2=\min\left\{||\Ric^{\gamma}_{\ippd}||^2:\mbox{\begin{tabular}{l} $\ippd$ is a compatible metric with  $(\GrpN_{\mu},\gamma)$ \\
and $\scalar(\ippd)=\scalar(\ipp)$
\end{tabular} } \right\}
$$
\end{definition}

Without loss of generality, from now on we can assume that the canonical inner product of $\RR^n$ (which we denoted also by $\ip$) is compatible with $\gamma$ and that $\mathrm{G}_{\gamma}$ is self-adjoint with respect to this one. Therefore $\mathrm{G}_{\gamma}$ is compatible with the usual Cartan decomposition of $\Glr$: $\mathrm{G}_{\gamma}=\mathrm{K}_{\gamma}\exp(\mathfrak{p}_{\gamma})$  where $\mathrm{K}_{\gamma} \subseteq \mathrm{O}(n)$ is a maximal compact subgroup of $\mathrm{G}_{\gamma}$ and $\mathfrak{p}_{\gamma} \subseteq \sym(n)$ is a vector subspace of symmetric matrices. It follows that $\mm_{\mathfrak{g}_{\gamma}}(v)$ is the orthogonal projection of $\mm_{\ggo_{\gamma}}(v)$ onto $\mathfrak{p}_{\gamma}$ for all $v \in V\smallsetminus\{ 0\}$.

\begin{proposition}\cite[Propositon 4.2]{LAURET3}
Let $(\mathrm{N}_{\mu},\gamma)$ be a class-$\gamma$ nilpotent Lie group. Then
\begin{eqnarray}
\mm_{\g}(\mu)&=&4\Ric_{\mu}, \\
\mm_{\mathfrak{g}_{\gamma}}(\mu) &=& 4 \Ric ^{\gamma}_{\mu},
\end{eqnarray}
where $\Ric_{\mu}$ is the Ricci operator of the Riemannian manifold $(\mathrm{N}_{\mu},\ip)$ and $\Ric^{\gamma}_{\mu}$ is the invariant Ricci operator of $(\mathrm{N}_{\mu},\gamma, \ip)$.
\end{proposition}

\begin{theorem}\cite[Proposition 4.3 and 4.4]{LAURET3}\label{minimalmetric}
Let $(\mathrm{N}_{\mu},\gamma)$ be a class-$\gamma$ nilpotent Lie group. $(\mathrm{N}_{\mu},\gamma)$ admits a minimal compatible metric if and only if the $\GrpG_{\gamma}$-orbit of $\mu$ is distinguished for the natural action of $\GrpG_{\gamma}$ on $V$. Moreover, there is at most one minimal compatible metric on $(\mathrm{N},\gamma)$ up to isometry (and scaling)
\end{theorem}

We are now in a position to apply previous results to study the existence of minimal metrics on a special family of nilpotent Lie algebras. We denote by $\ag_{\gamma}:=\ag \cap \pg_{\gamma}$ an abelian subalgebra of $\ggo_{\gamma}$ maximal in $\pg_{\gamma}$ and $\mathrm{A}_{\gamma}:=\exp(\ag_{\gamma})$. Since we are considering many group actions and their respective moment maps, it is convenient to introduce the following definition.

\begin{definition}
Let $\mathrm{G}_{\gamma}$ be a reductive subgroup of $\Glr$ compatible with the usual Cartan decomposition of $\Glr$ and let $W$ be an $\mathrm{A}_{\gamma}$-invariant vector subspace of $V:=\lamn$. We call $W$ \textit{$A_{\gamma}$-nice} if $W$ is nice with respect to the action of $\mathrm{G}_{\gamma}$ on $V$.
\end{definition}

\begin{notation}
Let us denote by $\mathfrak{R}_{\gamma}(\mu)$ denote the ordered set of weights related with $\mu$ to the action of $\mathrm{G}_{\gamma}$ on $V$. It is clear that $\mathfrak{R}_{\gamma}(\mu)= \Proj_{\ag_{\gamma}}\mathfrak{R}(\mu)$ where $\mathfrak{R}(\mu)$ are weights related with $\mu$ to the action of $\Glr$ on $V$.
\end{notation}

\begin{theorem}\label{nicegeneralizado}
Let $W$ be a $\mathrm{A}_{\gamma}$-nice space and let $(\mathrm{N}_{\mu},\gamma)$ be a class-$\gamma$ nilpotent Lie group with $\mu \in W$. $(\mathrm{N}_{\mu},\gamma)$ admits a compatible minimal metric if and only if the equation
$$
\gramU^{\gamma}_{\mu}[x_i]=\lambda[1]
$$
has a positive solution $[x_i]$ for some $\lambda \in \RR$. Here, $\gramU^{\gamma}_{\mu}$  is the Gram matrix of $(\mathfrak{R}_{\gamma}(\mu),\ipd)$
\end{theorem}

\begin{example}
We consider the $6$-dimensional nilpotent Lie algebra $\ngo:=(\mu,\RR^n)$ given by the direct sum of two $3$-dimensional Heisenberg Lie algebra, i.e.
$$
\mu:=\left\{ [e_1,e_4]=e_6, \, [e_2,e_3]=e5 \right.
$$
$\ngo$ admits only two symplectic structures up to symplecto-isomorphism (see, for instance, \cite[Theorem 5. 24]{KHAKIMDJANOV1})
$$
\omega(\pm)= \pm(e_{1}^{*}\wedge e_{6}^{*} + e_{2}^{*}\wedge e_{5}^{*} + e_{3}^{*}\wedge e_{4}^{*})
$$
We will now prove that $(\mathrm{N}_{\mu},\omega(\pm))$ admit a compatible minimal metric.

It is easy to see that the canonical metric of $\RR^6$, $\ip$, defines a compatible metric with $\omega(\pm)$ and $\mathrm{G}_{\omega(\pm)}=\mathrm{Sp}(3,\RR)$. Any compatible metric with $\omega(\pm)$ is in the $\mathrm{Sp}(3,\RR)$-orbit of $\ip$, hence we must minimize the function $||\Ric^{\omega}_{(\cdot,\cdot)}||^2$ over all $\mathrm{Sp}(3,\RR) \cdot \ip$, which is equivalent to minimize $||\mm_{\spg(6,\RR)}||^2$ over the $\mathrm{Sp}(3,\RR)$-orbit of $\mu$. By Theorem \ref{stratification} item \ref{teocrit}, we must show that $\mathrm{Sp}(3,\RR)\cdot \mu$ is $\mathrm{Sp}(3,\RR)$-distinguished.

Let
$$
W=\spanv_{\RR} \{ \mu_{14}^6, \mu_{23}^5\}
$$
where $\mu_{ij}^{k}$ is the bracket defined as $\mu_{ij}^{k}(e_i,e_j)=e_k=-\mu_{ij}^{k}(e_j,e_i)$ and zero in otherwise.
Let us first see that $W$ is a $\mathrm{A}_{\omega}$-nice. A way of making this is to note that $W$ is a nice space and since $\Proj_{\pg_{\omega}} \ag \subseteq \ag_{\omega}$, it follows that $W$ is $\ag_{\omega}$-nice.

Other way of proving the same is using Corollary \ref{niceresta}. The root set $\Delta(\mathfrak{sp}(6,\RR))$ of is given by
$${\small\left\{
\begin{array}{ccc}
\pm\diag(1,0,0,0,0,-1), &
\pm\frac{1}{2} \diag(1,1,0,0,-1,-1), &
\pm\frac{1}{2} \diag(1,-1,0,0,1,-1), \\
\pm\diag(0,1,0,0,-1,0), &
\pm\frac{1}{2} \diag(1,0,1,-1,0,-1), &
\pm\frac{1}{2} \diag(1,0,-1,1,0,-1), \\
\pm\diag(0,0,1,-1,0,0), &
\pm\frac{1}{2} \diag(0,1,1,-1,-1,0), &
\pm\frac{1}{2} \diag(0,1,-1,1,-1,0) \\
\end{array}
\right\}}
$$
The Weights of $W$ with respect to the action of $\mathrm{Sp}(3,\RR)$ are
$$\{ \alpha_1:=\diag(-1,0,\tfrac{1}{2},-\tfrac{1}{2},0,1), \alpha_2:=\diag(0,-1,-\tfrac{1}{2},\tfrac{1}{2},1,0)\}
$$
Since $\alpha_1-\alpha_2 \notin \Delta({\mathfrak{sp}(3,\RR)})$, it follows that $W$ is $\ag_{\omega}$-nice.

In the light of the Theorem \ref{onedirection}, we need to show that $\mcc({\alpha_1,\alpha_2}) \in \inte(\CH(\{\alpha_1,\alpha_2\}))$.
$$
\gramU^{\gamma}_{\mu} = \left(
                              \begin{array}{cc}
                                \frac{5}{2} & -\frac{1}{2} \\
                                -\frac{1}{2} & \frac{5}{2} \\
                              \end{array}
                            \right)
$$
Since $X=(\tfrac{1}{2},\tfrac{1}{2})$ is a positive solution to the problem $\gramU^{\gamma}_{\mu} = [1]_2$, it follows that $(\mathrm{N}_{\mu},\omega(\pm))$ admit a minimal metric.

To find such metric, we solve the problem
\begin{equation}\label{eqex}
\mm_{\mathfrak{sp}(3,\RR)} (\exp(X)\cdot \mu)= \mcc(\{\alpha_1,\alpha_2\})=\diag(-\tfrac{1}{2},-\tfrac{1}{2},0,0,\tfrac{1}{2},\tfrac{1}{2}) \end{equation}
with $X \in \ag_{\omega}$; $X=\diag(a,b,c,-c,-b,-a)$. Let us denote by $\widetilde{\mu}$ to $\exp(X)\cdot\mu$
$$
\widetilde{\mu}=\left\{[e_1,e_4]=\frac{\e^{c}}{\e^{2a}}e_6, \, [e_2,e_3]=\frac{1}{\e^{2b+c}}e_5\right.
$$
We have $\mm_{\mathfrak{sp}(3,\RR)}(\widetilde{\mu})$ is given by
$$
\diag(-2e^{2c-4a},-2e^{-4b-2c},e^{2c-4a}-e^{-4b-2c},e^{-4b-2c}-e^{2c-4a},2e^{-4b-2c},2e^{2c-4a})
$$
Setting $X=(\ln(2),0,\ln(2),-\ln(2),0,-\ln(2))$, we can solve the equation (\ref{eqex}), hence $\widetilde{\mu}=\frac{1}{2}\mu$ is a critical point of $\F_{\mathfrak{sp}(3,\RR)}$ which defines a minimal metric of $(\mathrm{N}_{\mu},\omega(\pm))$.
\end{example}

\begin{remark}
The first infinite families of symplectic Lie algebras appear in dimension $6$. In a forthcoming paper, we develop analogous results to those in \cite{NIKOLAYEVSKY2} for the symplectic case and these are used to classify minimal metrics on $6$-dimensional symplectic Lie algebras.
\end{remark}

Proceeding in an entirely analogous way, we can study the remainder symplectic two-step Lie algebras given in \cite[Theorem 5. 24]{KHAKIMDJANOV1} (up to 18.(b$_{t}$) and 18.(c), which were studied with similar methods used in \cite[Example 1.]{FERNANDEZ-CULMA1}).

\begin{theorem}\label{symtwo}
All symplectic two-step Lie algebras of dimension $6$ admit a minimal compatible metric.
\end{theorem}

\begin{remark}
We must say that we have found several mistakes in the classification given in \cite{KHAKIMDJANOV1}. For instance, 16.(b) does not define a symplectic structure. Some errors have already been corrected by personal communication with authors; as the symplectic structure given in 23.(c).
\end{remark}

In the Table \ref{minmetricas}, each Lie algebra defines a symplectic two-step Lie algebra given by $(\RR^6,\mu,\omega_{cn})$ where
$$
\omega_{cn}:=e_1^{\ast}\wedge e_6^{\ast} + e_2^{\ast} \wedge e_5^{\ast} + e_3^{\ast} \wedge e_4^{\ast},
$$
and is such that the canonical inner product on $\RR^6$ defines a minimal metric for every $(\RR^6,\mu,\omega_{cn})$. In the column $||\beta||^2$ we give the norm squared of the stratum associated to the minimal metric and in Derivation column, we give the derivation of $(\RR^6,\widetilde{\mu})$ such that
$$
\mm_{\mathfrak{sp}_{6}(\RR)}(\widetilde{\mu})=-||\beta||^2\Id + \mbox{Derivation}.
$$
In the last column, We give the dimension of automorphism group of the symplectic two-step Lie algebra $(\RR^6,\widetilde{\mu},\omega_{cn})$.

\begin{table}[h]
\begin{center}
\begin{tabular}{|c|l|c|c|c|}
\hline
\multicolumn{1}{ |c| }{\centering Not. } &
\multicolumn{1}{  c|    }{\centering Critical point } &
\multicolumn{1}{  c|  } {\centering {Derivation}} &
\multicolumn{1}{  m{0.6cm}|  }{\centering $||\beta||^2$ } &
\multicolumn{1}{  m{0.5cm }|  }{\centering $\dim$\\$\Aut$}
\tabularnewline
\hline
\hline
\multirow{2}{*}{16.(a) }&{$[{e_1},{e_2}]=\frac{\sqrt {2}}{4}\,{ e_3},\,[{ e_1},{ e_5}]=\frac{\sqrt {2}}{4}{ e_6},$}&\multirow{2}{*}{$\frac{1}{2}\diag(1, 2, 3, 1, 2, 3)$} & \multirow{2}{*}{$1$} & \multirow{2}{*}{$6$}\\
 &$[{ e_2},{e_4}]=\frac{\sqrt {2}}{4}{ e_6},[{ e_4},{e_5}]=\frac{\sqrt {2}}{4}{e_3}$ & & &  \\
\hline
\multirow{2}{*}{17.} & $[{e_1},{e_3}]=\frac{\sqrt {6}}{6}\,{e_5},\,[{ e_1},{e_4}]=\frac{\sqrt {6}}{6}{e_6},$ &\multirow{2}{*}{$\frac{1}{6}\diag(3, 5, 6, 8, 9, 11)$} &\multirow{2}{*}{$\frac{7}{6}$ }&\multirow{2}{*}{ $7$ }\\
 &$[{e_2},{e_3}]=\frac{\sqrt {6}}{6}{e_6}$ & & & \\
\hline
\multirow{3}{*}{18.(a$_{t}$)}  &$[{e_1},{e_2}]=\frac{1}{2}\,\sqrt {{\frac { \left( t-1 \right) ^{2}}{{t}^{2}-t+1}}}{e_4},$ &\multirow{3}{*}{$\diag(1, 1, 1, 2, 2, 2)$} & \multirow{3}{*}{$\frac{3}{2}$} &\multirow{3}{*}{ $8$}\\
              &$[{e_1},{e_3}]=\frac{t}{2(t-1)}\,\sqrt {{\frac { \left( t-1 \right) ^{2}}{{t}^{2}-t+1}}}{e_5},$&&&\\
              &${[{e_2},{e_3}]}=\frac{1}{2(t-1)}\,\sqrt {{\frac { \left( t-1 \right) ^{2}}{{t}^{2}-t+1}}}{e_6}$&&&\\
\hline
\multirow{3}{*}{18.(b$_{t}$)} &$[{e_1},{e_2}]={\frac {t}{\sqrt {3\,{t}^{2}+1}}}{e_4},$ & \multirow{3}{*}{$\diag(1,1,1,2,2,2)$} &\multirow{3}{*}{ $\frac{3}{2}$} & \multirow{3}{*}{$8$}\\
                              &$[{e_1},{e_3}]={\frac {t}{2\sqrt {3\,{t}^{2}+1}}}{e_5}+{\frac {1}{2\sqrt {3\,{t}^{2}+1}}}{e_6}, $ &&&\\
                              &${[{e_2},{e_3}]}={\frac {{1}}{2\sqrt {3\,{t}^{2}+1}}} {e_5}-{\frac {t}{2\sqrt{3\,{t}^{2}+1}}}{\e_6} $ &&&\\
\hline
\multirow{3}{*}{18.(c)}  &$[{e_1},{e_2}]=-\frac{\sqrt {3}}{12}{e_4}-\frac{\sqrt {3}}{4}{e_5},$ & \multirow{3}{*}{ $\diag(1,1,1,2,2,2)$ }& \multirow{3}{*}{$\frac{3}{2}$} & \multirow{3}{*}{$10$}\\
 &$[{e_1},{e_3}]=\frac{\sqrt {3}}{4}{e_4} + \frac{\sqrt {3}}{12}{e_5},$ & & & \\
 &$[{e_2},{e_3}]=\frac{\sqrt {3}}{6}{e_6}$ & & & \\
\hline
23.(a) &$[{e_1},{e_2}]=\frac{1}{2}\,{e_5},\,[{e_1},{e_3}]=\frac{1}{2}{e_6}$ &$\frac{1}{4}\diag(4, 5, 6, 8, 9, 10)$ &$\frac{7}{4}$ &$9$ \\
\hline
23.(b) &$[{e_1},{e_2}]=-\frac{1}{2}\,{e_4},\,[{e_2},{e_3}]=\frac{1}{2}\,{e_6}$ &$ \diag(1, 1, 1, 2, 2, 2)$ &$\frac{3}{2}$ & $8$\\
\hline
23.(c) & $ [{e_1},{e_2}]=\frac{1}{2}\,{e_5},\,[{e_1},{e_4}]=\frac{1}{2}\,{e_3} $ & $\diag(1, 1, 2, 1, 2, 2)$ & $\frac{3}{2} $& $ 8$ \\
\hline
24.(a) & $ [{e_1},{e_4}]=\frac{1}{2}\,{e_6},\,[{e_2},{e_3}]=\frac{1}{2}\,{e_5} $ & $2\diag(1, 1, 2, 2, 3, 3)$ & $1$ & $ 6$ \\
\hline
24.(b) & $ [{e_3},{e_6}]=-\frac{1}{2}\,{e_1},\,[{e_4},{e_5}]=-\frac{1}{2}\,{e_2} $ & $\frac{1}{2}\diag(3, 3, 2, 2, 1, 1)$ & $1$ & $6$ \\
\hline
25. & $ [{e_1},{e_2}]=\frac{\sqrt {2}}{2}\,{e_6}$ & $ \frac{1}{2}\diag(3, 4, 5, 5, 6, 7)$ & $ \frac{5}{2} $ & $ 12$ \\
\hline
\end{tabular}
\end{center}
\caption{Classification of minimal compatible metrics on symplectic two-step Lie algebras of dimension $6$}\label{minmetricas}
\end{table}

\begin{remark}
The above minimal metrics define \textit{soliton solutions} to the \textit{Street-Tian symplectic curvature flow}, which was recently introduced in \cite{STREETS1}. To be more precise, following Vezzoni \cite{VEZZONI1}, in \cite{POOK1} has been noted that, in two-step nilpotent Lie groups, the \textit{symplectic curvature flow} reduces to the \textit{anti-complexified Ricci flow}, which was given by H\^{o}ng V\^{a}n L\^{e} and  Guofang Wang in \cite{LE1}. Given that the invariant Ricci tensor $\ricci^{\omega}_{\ip}$ coincides with the \textit{anti-complexified Ricci tensor} $\ricci^{ac}_{\ip}$ (see \cite[Equation (23)]{LAURET3}), the affirmation follows from \cite[Proposition 2.7]{LAURET3}.
\end{remark}

\section{Appendix}

\subsection{Convex Functions and Convex Geometry}
In this part, let us recall some notions and basic results of convex functions and convex geometry which were needed for the proof of Propositions \ref{mmAorbit} and \ref{niceelement}.

\begin{definition}\cite[Definitions in V.3]{NEEB1}
Let $V$ be a real vector space and let $\fnn{f}{V}{\RR_{\infty}}$ be a convex function. The set $D_{f}:=f^{-1}(\RR)$ is called the \textit{domain} of $f$.

The function
\begin{equation*}
\fna{f^{*}}{V^{*}}{\RR_{\infty}}{\xi}{\sup_{V}(\xi - f)}
\end{equation*}
is called the conjugate of $f$.
\end{definition}

\begin{theorem}\label{Fenchel}\cite[Corollary V.3.32]{NEEB1}(\textbf{Fenchel's Convexity Theorem})
Let $f\in \mathscr{C}^2(V)$ be such that $\dif^2 f (X)$ is positive definite for all $X \in V$. Then $\dif f$ maps $V$ diffeomorphically onto the open convex set $\inte D_{f^{*}}$ and $\fnn{\dif f^{*}}{ \inte D_{f^{*}} }{ V}$ is the inverse of $\dif f$.
\end{theorem}

\begin{definition}\cite[Definitions in 2.4]{GRUNBAUM1}
Let $\Omega$ be a convex subset of $\RR^n$. A point $X \in \Omega$ is an \textit{extreme point} of $\Omega$ if it does not belong to the relative interior of any segment contained in $\Omega$, e.d. $X$ is an extreme point if $X= tY + (1-t)Z$ with $0<t<1$ and $Y$ and $Z$ in $\Omega$ implies $X=Y=Z$. Intuitively, an extreme point is a \comillas{corner} of $\Omega$. The set of all extreme points of $\Omega$ is denoted by $\xt(\Omega)$.

A point $X \in \Omega$ is an \textit{exposed point} of $\Omega$ if there exists a supporting hyperplane of $\Omega$ whose intersection with $\Omega$ is $X$ alone (Fig. \ref{exposed}); we mean, there exists a hyperplane of $\RR^n$, namely $\Pi=\{Y \in \RR^n : \ipda{Y}{H}=h \}$ with $H\in \RR^n$, $h \in \RR$ and $\ipd$ is a inner product in $\RR^n$, such that $\ipda{Y}{H}\geq h$ for all $Y \in \Omega$ and the equality holds if and only if $X=Y$. The set of all exposed points of $\Omega$ is denoted by $\xp(\Omega)$.
\end{definition}

It is easily seen that $\xp(\Omega) \subseteq \xt(\Omega)$ for all convex set $\Omega$.

\begin{theorem}\label{expTh}\cite[Theorems 2.4.5 and 2.4.9]{GRUNBAUM1}
\begin{enumerate}
\item (\textbf{Minkowski-Krein-Milman Theorem}) Let $\Omega$ be a compact convex subset of $\RR^n$. Then $\Omega = \CH(\xt(\Omega))$. Moreover, if $\Omega = \CH(\Phi)$ then $\Phi \supseteq \xt(\Omega)$.
\item (\textbf{Straszewicz's Theorem})  If $\Omega$ is a closed convex set then $\xt(\Omega) \subseteq \overline{\xp(\Omega)}$.
\item \label{politope} If $\Omega = \CH(\Phi)$ with $\Phi$ a finite subset of $\RR^n$ ($\Omega$ is a \textit{polytope}) then $\Omega=\CH(\xp(\Omega))$ and $\Phi \supseteq \xp(\Omega)$.
\end{enumerate}
\end{theorem}

\begin{figure}
  \includegraphics[width=7cm]{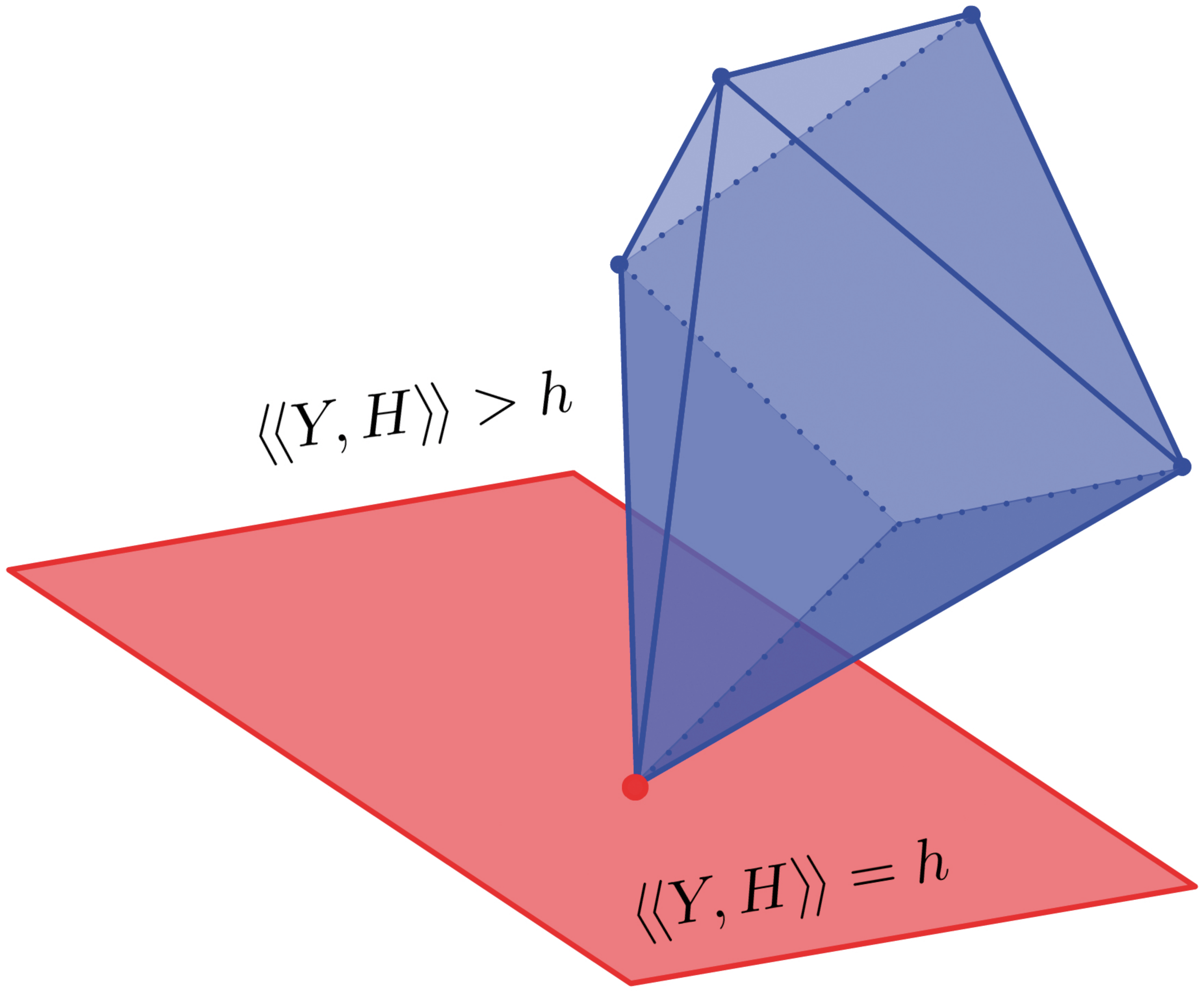}\\
  \caption{Exposed point}\label{exposed}
\end{figure}

\subsection{Canonical forms for binary forms of low degree}

We recall canonical forms for binary forms of degree $2$, $3$ and $4$.

\begin{table}
\centering
\begin{tabular}{|c|l|c|l|}
\hline
\multicolumn{4}{  | c |          }  {\centering Canonical Form }          \tabularnewline
\multicolumn{4}{  | c |           }  {\centering Binary quartic forms }          \tabularnewline
\hline
 Not. & Over $\CC$ & Not. & Over $\RR $ \\
\hline
\multirow{8}{*}{ I } & \multirow{8}{*}{ $x^4 + 2tx^2y^2+y^4$ } &\multirow{2}{*}{ a.} &  \multirow{1}{*}{ $x^4 + 2tx^2y^2+y^4$}  \\
 &  \multirow{8}{*}{$(t \neq 1)$ }& & $(-1 < t \neq 1)$\\
\cline{3-4}
                     & &\multirow{2}{*}{ b.} & $-[x^4 + 2tx^2y^2+y^4]$ \\
 & & & $(-1 < t \neq 1)$ \\
\cline{3-4}
                     & & \multirow{2}{*}{c.} & $x^4 + 2tx^2y^2-y^4$  \\
 & & & $(t \in \RR)$\\
\cline{3-4}
                     & & \multirow{2}{*}{d.} & $x^4 + 2tx^2y^2+y^4$ \\
 & & & $(-1 > t)$\\
\hline
\multirow{4}{*}{ II } & \multirow{4}{*}{ $x^2 y^2 + y^4$ }& a. &  $x^2 y^2 + y^4$ \\
\cline{3-4}
                     & & b. & $-[x^2 y^2 + y^4]$  \\
\cline{3-4}
                     & & c. & $x^2 y^2 - y^4 $    \\
\cline{3-4}
                     & & d. & $-[x^2 y^2 - y^4]$  \\
\hline
\multirow{4}{*}{ III } & \multirow{4}{*}{  $x^2 y^2$ } & a. & $x^2 y^2$ \\
\cline{3-4}
                     & & b. & $-[x^2 y^2]$ \\
\cline{3-4}
                     & & c. & $x^4+2x^2 y^2+y^4$  \\
\cline{3-4}
                     & & d. & $-[x^4+2x^2 y^2+y^4]$ \\
\hline
                 IV    &  $x^3 y$  &  a.  & $x^3 y$ \\
\hline
\multirow{2}{*}{ V }   & \multirow{2}{*}{ $x^4$ }   &  a.  & $x^4$ \\
\cline{4-4}
                       &                            &  b.  & $-x^4$ \\
\hline
                VI     &   $0$                        &  a.  & $0$ \\
\hline
\end{tabular}
\caption{Canonical form for binary quartic forms}\label{table4t}
\end{table}

\begin{table}[h]
\centering
\begin{tabular}{|c|l|c|l|}
\hline
\multicolumn{4}{  | c |          }  {\centering Canonical form }          \tabularnewline
\multicolumn{4}{  | c |           }  {\centering Binary cubic forms}          \tabularnewline
\hline
 Tipo & Sobre $\CC$ & Tipo & Sobre $\RR $ \\
\hline
\multirow{2}{*}{ I }   &\multirow{2}{*}{ $y^3 + yx^2$ }   & a. & $y^3 + yx^2$ \\
\cline{3-4}
                       &                                  & b. & $y^3 - yx^2$ \\
\hline
 II                  & $y^2 x$  &  a. &   $y^2 x$ \\
\hline
 III                 & $y^3$    &  a. &   $y^3$   \\
\hline
IV                   &  $0$     &  a. &    $0$    \\
\hline
\end{tabular}
\caption{Canonical forms for binary cubic forms}\label{table3n}
\end{table}

\begin{table}[h]
\centering
\begin{tabular}{|c|l|c|l|}
\hline
\multicolumn{4}{  | c |          }  {\centering Canonical form}          \tabularnewline
\multicolumn{4}{  | c |           }  {\centering Binary quadratic forms}          \tabularnewline
\hline
 Tipo & Sobre $\CC$ & Tipo & Sobre $\RR $ \\
\hline
\multirow{3}{*}{ I } & \multirow{3}{*}{$x^2+y^2$} & a. & $x^2+y^2$ \\
\cline{3-4}
                     &                            & b. & $-[x^2+y^2]$ \\
\cline{3-4}
                     &                            & c. & $x^2 - y^2$ \\
\hline
\multirow{2}{*}{II}  &\multirow{2}{*}{$x^2$ }     & a. & $x^2$ \\
\cline{3-4}
                     &                            & b. & $-x^2$ \\
\hline
 III                 &  $0$                       & a. & $0$\\
\hline
\end{tabular}
\caption{Canonical forms for binary quadratic forms}\label{table2d}
\end{table}

\end{document}